\documentclass[10pt,a4paper]{article}
\usepackage[utf8]{inputenc}
\usepackage{amsmath}
\usepackage{amsfonts}
\usepackage{amssymb}
\usepackage{amsthm}
\usepackage{mathtools}
\usepackage{enumitem}
\usepackage{hyperref}

\title{Bound on the number of Ruelle resonances for Gevrey hyperbolic flows}
\author{Malo J\'ez\'equel\thanks{CNRS, Univ. Brest, UMR6205, Laboratoire de Math{\'e}matiques de Bretagne Atlantique, France. email: malo.jezequel@math.cnrs.fr}}
\date{}

\newtheorem{theorem}{Theorem}
\newtheorem{lemma}{Lemma}[section]
\newtheorem{proposition}[lemma]{Proposition}
\theoremstyle{definition}
\newtheorem{definition}[lemma]{Definition}
\theoremstyle{remark}
\newtheorem{remark}[lemma]{Remark}

\newcommand{\set}[1]{\left\{#1\right\}}
\newcommand{\n}[1]{\left\| #1 \right\|}
\newcommand{\st}{\textup{st}}
\newcommand{\un}{\textup{un}}

\newcommand{\brac}[1]{\left\langle #1 \right\rangle}

\DeclareMathOperator{\tr}{tr}
\DeclareMathOperator{\re}{Re}
\DeclareMathOperator{\Un}{Un}

\DeclareMathOperator{\Fix}{Fix}
\DeclareMathOperator{\length}{length}
\DeclareMathOperator{\supp}{supp}
\DeclareMathOperator{\res}{Res}

\begin{document}

\maketitle

\begin{abstract}
We improve the best known upper bounds on the number of Ruelle resonances in disks of large radius for Gevrey uniformly hyperbolic flows. The proof is based on Rugh's approach of dynamical determinants \cite{rugh_axiom_a} that replaces the study of the flow itself by the analysis of a system of open hyperbolic maps.
\end{abstract}

\section{Introduction}

Smooth uniformly hyperbolic flows on compact manifolds are called \emph{Anosov}. This class of flows has been extensively studied (in particular for geometric motivation: the geodesic flow on the unit tangent bundle of a compact Riemannian manifold of negative sectional curvature is Anosov) and is well-known for its chaotic properties. For instance, if $\phi = (\phi_t)_{t \in \mathbb{R}}$ is a $C^2$ Anosov flow on a smooth manifold $M$ that preserves a smooth volume form $\mathrm{d}x$ and whose stable and unstable foliations are not jointly integrable then it is \emph{mixing} \cite{sinai_geodesic,anosov_sinai}: for every $f,g \in L^2(\mathrm{d}x)$ we have
\begin{equation}\label{eq:mixing}
\int_{M} (f \circ \phi_t) g \mathrm{d}x \underset{t \to + \infty}{\to} \int_M f \mathrm{d}x \int_M g \mathrm{d}x.
\end{equation}
It is then natural to wonder if the convergence \eqref{eq:mixing} may be upgraded into an asymptotic expansion. Some answer to this question is given by the theory of \emph{Ruelle resonances}. If $\phi = ( \phi_t)_{t \in \mathbb{R}}$ is now a $C^\infty$ Anosov flow on a smooth manifold $M$ and $\mathrm{d}x$ a smooth volume form on $M$, then there is a discrete subset $\res(\phi)$ of $\mathbb{C}$ such that for every $f,g \in C^\infty(M)$ we have
\begin{equation}\label{eq:expansion_resonances}
\int_M (f \circ \phi_t) g \mathrm{d}x \underset{t \to + \infty}{\approx} \sum_{\lambda \in \res(\phi)} P_{\lambda,f,g}(t) e^{\lambda t},
\end{equation}
where the $P_{\lambda,f,g}$ are polynomials, see \cite{butterley_liverani, butterley_liverani_corrigendum, faure_sjostrand}. Beware that the sign $\approx$ has a very weak sense in \eqref{eq:expansion_resonances}: we just mean that the Laplace transform of the left hand side has a meromorphic extension to $\mathbb{C}$ whose poles are contained in $\res(\phi)$. Proving that the approximation \eqref{eq:expansion_resonances} holds in a stronger sense is a difficult question (see \cite{faure_tsujii_band_structure}), which is related to the problem of establishing a rate of mixing for Anosov flows (see e.g. \cite{liverani_contact}).

The goal of this article is to understand the distribution of \emph{Ruelle resonances} (the elements of the set $\res(\phi)$ from \eqref{eq:expansion_resonances}). More precisely, we care about the growth when $r$ goes to $+ \infty$ of the counting function:
\begin{equation}\label{eq:counting_function}
N(r) = \# \set{\lambda \in \res(\phi) : |\lambda| \leq r}.
\end{equation}
Understanding the asymptotic of the counting function is an old problem in scattering theory \cite{regge_scattering}. The problem that we consider here is an analogue in the context of hyperbolic dynamics. In order to study the growth of \eqref{eq:counting_function}, it is interesting to introduce the \emph{dynamical determinant} defined for $\re z \gg 1$ by 
\begin{equation}\label{eq:determinant_simple}
d(z) = \exp\left( - \sum_{\gamma \in \mathfrak{P}} \frac{T_\gamma^{\#}}{T_\gamma}\frac{e^{- z T_\gamma}}{|\det(I - \mathcal{P}_\gamma)|} \right).
\end{equation}
Here, $\mathfrak{P}$ denotes the set of periodic orbits\footnote{We define a periodic orbit for a flow $\phi$ as a finite length flow line for $\phi$ whose endpoints coincide. We identify two periodic orbits if they have the same image and the same length.} of $\phi$ and if $\gamma \in \mathfrak{P}$, we write $T_\gamma$ for its length, $T_\gamma^\#$ for its primitive length and $\mathcal{P}_\gamma$ for the associated linearized Poincaré map\footnote{See \S \ref{subsection:hyperbolic_flows} for a precise definition of these objects.}. The dynamical determinant $d(z)$ has a holomorphic continuation to $\mathbb{C}$ whose zeroes are exactly the Ruelle resonances \cite{glp_zeta,dyatlov_zworski_zeta, dyatlov_guillarmou_open_systems,dyatlov_guillarmou_afterword}. Hence, according to Jensen's formula, any bound on $|d(z)|$ when $|z|$ goes to $+ \infty$ implies a bound on $N(r)$ when $r$ goes to $+ \infty$ (and the converse is not that far from being true).

Some counter-examples \cite{jezequel_trace, tao_counting} suggest that there is no general upper bound on \eqref{eq:counting_function} in the asymptotic $r \to + \infty$. However, one can get an upper bound on the counting function by imposing a regularity assumption on $\phi$ (stronger than $C^\infty$). Indeed, for a real-analytic Anosov flow the bound $N(r) \underset{r \to + \infty}{=} \mathcal{O}(r^n)$, where $n$ is the dimension of the manifold the flow is acting on, follows from the work of Rugh and Fried \cite{rugh_correlation_spectrum, rugh_axiom_a, fried_zeta, fried_zeta2}. A bound on the number of resonances for a class of ultradifferentiable Anosov flows is given in \cite{jezequel_ultradifferentiable}. In \cite{asterisque}, it is proven that if $\phi$ is a $s$-Gevrey\footnote{For some $s \geq 1$. See \S \ref{subsection:Gevrey} for the definition of Gevrey regularity.} Anosov flow on a manifold of dimension $n$ then $N(r) \underset{r \to + \infty}{=} \mathcal{O}(r^{ns})$. Our first result is an improvement of this bound.

\begin{theorem}\label{theorem:resonances_simple}
If $s > 1$ and $\phi$ is a $s$-Gevrey Anosov flow on a manifold of dimension $n$, then $N(r) \underset{r \to + \infty}{=} \mathcal{O}(r^{1 + (n-1)s}).$
\end{theorem}

Our second result is the dynamical determinant counterpart of Theorem \ref{theorem:resonances_simple}.

\begin{theorem}\label{theorem:determinant_simple}
If $s > 1$ and $\phi$ is a $s$-Gevrey Anosov flow on a manifold of dimension $n$, then there is a constant $C > 0$ such that for every $z \in \mathbb{C}$ we have
\begin{equation*}
|d(z)| \leq \begin{cases} C \exp(C|z|^{1 + (n-1)s}) & \textup{ if } 1 + (n-1)s \textup{ is not an integer, } \\ C \exp(C|z|^{1 + (n-1)s} \log(1+|z|)) & \textup{ if } 1 + (n-1)s \textup{ is an integer. }\end{cases}
\end{equation*}
\end{theorem}

Theorems \ref{theorem:resonances_simple} and \ref{theorem:determinant_simple} are simpler versions of more general statements (Theorems \ref{theorem:bound_resonances} and \ref{theorem:bound_determinant}). Notice that, in Theorems \ref{theorem:resonances_simple} and \ref{theorem:determinant_simple}, we excluded the case $s = 1$ (that corresponds to real-analytic regularity). It will be convenient for their proofs (for instance, there are Gevrey partitions of unity when $s > 1$). Moreover, these bounds are already known in the case $s = 1$, see \cite[Theorem 3.1 and Proposition 3.2]{asterisque}. 

\begin{remark}
The more general Theorems \ref{theorem:bound_resonances} and \ref{theorem:bound_determinant} deal with the case of basic sets for a hyperbolic flows. Another important class of flow is the class of Axiom A flow \cite{smale_differentiable}, for which a theory of Ruelle resonances is developed in \cite{meddane_morse}. The results in Theorem \ref{theorem:resonances_simple} and \ref{theorem:determinant_simple} are also valid for Gevrey Axiom A flow. Indeed, the nonwandering set for an Axiom A flow may be decomposed into a finite number of hyperbolic fixed point and a finite number of basic sets \cite[Part II, Theorem 5.2]{smale_differentiable}. The dynamical determinant for an Axiom A flow is then the product of the dynamical determinants associated to each of its basic sets (a similar decomposition for the zeta function of an Axiom A flow is used in \cite{dyatlov_guillarmou_afterword}, here we do not count fixed point as periodic orbits), and the analogues of Theorem \ref{theorem:determinant_simple} for Axiom A flows is then a consequence of Theorem \ref{theorem:bound_determinant}. Concerning the bound on the number of resonances, we know that the set of resonances for an Axiom A flow is the union of the sets of resonances for its basic sets and fixed points \cite[Proposition 5.11]{tao_counting}. Basic sets are dealt with using Theorem \ref{theorem:bound_resonances} and resonances corresponding to fixed points may be computed explicitly \cite[Proposition 6.5]{tao_counting} (they only contribute to the number of resonances of modulus less than $r$ by a $\mathcal{O}(r^n)$ as $r$ goes to $+ \infty$), proving that the analogue of Theorem \ref{theorem:resonances_simple} for Axiom A flows also holds.
\end{remark}

\begin{remark}
From Theorem \ref{theorem:determinant_simple}, we deduce in particular that the dynamical determinant $d(z)$ associated to a $s$-Gevrey Anosov flow has order less than $1 + (n-1)s$. It is standard that the finiteness of the order of the dynamical determinant implies a \emph{trace formula} \cite[Proposition 1.5]{jezequel_ultradifferentiable}: a distributional equality that relates the periodic orbits of the flow with the resonances.  The trace formula is proven in \cite{jezequel_ultradifferentiable} for Anosov flows in a class of regularity larger than Gevrey. However, this reference does not deal with the cases of basic sets or Axiom A flows. As a consequence of Theorem \ref{theorem:bound_determinant}, it is now established that the trace formula holds for Gevrey Axiom A flows and for basic sets of Gevrey hyperbolic flows. Notice however that this trace formula will only take into account resonances corresponding to basic sets (ignoring those coming from fixed points). We expect the trace formula to hold for Axiom A and basic sets in the same class of regularity as in the Anosov case. 
\end{remark}

\subsection*{Strategy of proof}

The modern method to study Ruelle resonances for a hyperbolic flow $\phi$ consists in designing a family $(\mathcal{H}^a)_{a > 0}$ of (Banach or Hilbert) \emph{spaces of anisotropic distributions} such that for each $a > 0$ the intersection of $\set{z \in \mathbb{C} : \re z > - a}$ with the spectrum of the generator $X$ of the flow $\phi$ (seen as a differential operator of order $1$) acting on $\mathcal{H}^a$ is made of isolated eigenvalues of finite multiplicity. These eigenvalues are the Ruelle resonances of $\phi$. This method has been first applied in the continuous-time case by Butterley and Liverani \cite{butterley_liverani,butterley_liverani_corrigendum}. Another standard reference is \cite{faure_sjostrand} and recent textbook treatment of this topic may be found in \cite{lefeuvre_book}. One may also refer to \cite{baladi_book} for the discrete-time case.

This method is ill-suited to the study of the counting function $N(r)$. Indeed, for a given $a > 0$, the space $\mathcal{H}^a$ only gives us access to the resonances of $\phi$ in the half-space $\set{z \in \mathbb{C} : \re z >- a}$, missing a significant part of the disk $\set{ z \in \mathbb{C} : |z| \leq r}$ when $r$ is large. However, under some extra regularity assumption on $\phi$, one may try to construct a single Hilbert space $\mathcal{H}$ on which $X$ has compact resolvent. In particular, the spectrum of $X$ on $\mathcal{H}$ is made of isolated eigenvalues of finite multiplicites (that coincide with Ruelle resonances if $\mathcal{H}$ is well-designed). By establishing that the singular values of the resolvent of $X$ decays at a certain rate, one naturally gets an upper bound on $N(r)$ when $r$ goes to $+ \infty$. A variation on this method is applied in \cite{jezequel_ultradifferentiable} to study trace formulae for a class of ultradifferentiable Anosov flows\footnote{See also \cite{jezequel_trace,jezequel_circle,jezequel_distribution} for a similar approach in the discrete-time case.}.

The strategy described above is implemented in the Gevrey setting in \cite{asterisque}, but yields the bound $N(r) \underset{r \to + \infty}{=} \mathcal{O}(r^{ns})$ instead of $N(r) \underset{r \to + \infty}{=} \mathcal{O}(r^{1 +(n-1)s})$. Let us say a few words about the strategy of \cite{asterisque} in order to explain why it misses the bound from Theorem \ref{theorem:resonances_simple}. Consider an Anosov flow $\phi$ on a manifold $M$ with generator $X$. To the differential operator $X$, one may associate a function $p$ on the cotangent bundle $T^* M$, explicitly defined by $p(x,\xi) = i \xi (X(x))$, the \emph{symbol} of $X$. Microlocal analysis suggests that the number of eigenvalues for $X$ of modulus less than $r >0$ should be controlled\footnote{In the case of an elliptic self-adjoint operator, a much sharper results, called \emph{Weyl law}, is expected. See for instance \cite[\S 6.4]{zworski_book}.} by the volume\footnote{Recall that there is a canonical volume form on $T^* M$.} of the sublevel set $\set{|p| \leq r}$. However, since $p$ vanishes on a hyperplane in each fiber of the cotangent bundle, one may check that the volume $|\set{|p| \leq r}|$ is actually infinite. This is a manifestation of the lack of ellipticity of the operator $X$. The plan from \cite{asterisque} to bypass this issue is to let $X$ act on a Hilbert space $\mathcal{H}$ (instead of $L^2(M)$ for instance), which effectively replaces $p$ by a new symbol $\tilde{p}$. The space $\mathcal{H}$ is designed to make $p$ non-zero by adding a negative term of order of magnitude $|\xi|^{\frac{1}{s}}$ near $p = 0$. Consequently, we could expect $|\set{|\tilde{p}| \leq r}| \lesssim r^{1 + (n-1)s}$, and thus $N(r) \lesssim r^{1 + (n-1)s}$. However, for technical reasons this bound is not achieved. Indeed, to prove that $\tilde{p}$ is large on a \emph{conical} neighbourhood of $\set{p = 0}$, we need to use the correction term of size $|\xi|^{\frac{1}{s}}$ only, which yields $|\set{|\tilde{p}| \leq r}| \lesssim r^{ns}$. Hence, we are not able with the tools of \cite{asterisque} to use as much as we would like the original ellipticity of the symbol $p$ outside of $\set{p = 0}$. Even if $p$ is elliptic of order $1$ outside of a hyperplane, we can only use that $\tilde{p}$ is elliptic of order $1/s$ in a conical neighbourhood of $\set{p = 0}$.

What is at stake here is to capture the specific behaviour of the flow in its own direction. We will do it by ``factoring out'' the direction of the flow. Concretely, we will study a family of maps induced by the flow between disks transverse to the flow. This is a method that goes back to the work of Rugh \cite{rugh_correlation_spectrum, rugh_axiom_a} and then Fried \cite{fried_zeta, fried_zeta2} in the analytic category. We have also used this method recently in the smooth category to study zeta functions associated to pseudo-Anosov flows in a work with Jonathan Zung \cite{jezequel_zung}. This strategy is based on the understanding of the dynamical determinant \eqref{eq:determinant_simple}. Indeed, we will not study the action of the flow or its generator on a space of distributions. Instead, we will produce a family of trace class operators on some Hilbert spaces such that \eqref{eq:determinant_simple} is a quotient of products of determinants of these operators. These trace class operators are transfer operators associated to family of open hyperbolic maps in the spirit of \cite{rugh_axiom_a}. We will deal with the combinatorics of periodic orbits for the flow using a Markov partition. However, it is essential that we do not work with the transfer operator associated to the subshift of finite type that codes the flow: we would lose our regularity assumption on the flow by doing so. Instead, we work with a family of Gevrey maps indexed by the vertices of the graph associated to the Markov partition. The main difference between the present paper and \cite{rugh_correlation_spectrum,rugh_axiom_a,fried_zeta,fried_zeta2,jezequel_zung} are the spaces on which we let the transfer operator associated to our family of maps act. Our spaces are adapted to the Gevrey regularity (as in \cite[Appendix B]{jezequel_distribution}), and could be interpreted as spaces of ultradistributions, while the spaces of \cite{rugh_correlation_spectrum,rugh_axiom_a,fried_zeta,fried_zeta2} are adapted to analytic regularity. In \cite{jezequel_zung}, we work the spaces adapted to finite regularity from \cite{baladi_tsujii}.

A drawback of the method used in this paper is that we have access only to the dynamical determinant associated to an Anosov flow. A priori, we cannot say anything about the statistical properties of the flow. This is only because we know that the resonances are the zeroes of a dynamical determinant that we are able to prove Theorem \ref{theorem:resonances_simple}.

\subsection*{Structure of the paper}

In \S \ref{section:preliminaries}, we recall some basic definitions and results about hyperbolic flows (\S \ref{subsection:hyperbolic_flows}) and Gevrey regularity (\S \ref{subsection:Gevrey}). We also state in \S \ref{subsection:results} more general versions of Theorems \ref{theorem:resonances_simple} and \ref{theorem:determinant_simple}: Theorems \ref{theorem:bound_resonances} and \ref{theorem:bound_determinant}. The rest of the paper is dedicated to their proofs.

In \S \ref{section:periodic_orbit}, we prove Theorems \ref{theorem:bound_resonances} and \ref{theorem:bound_determinant} in the case of a basic set made of a single periodic orbit. Indeed, the strategy based on Markov partition that we use to prove Theorems \ref{theorem:bound_resonances} and \ref{theorem:bound_determinant} does not apply in that case. However, resonances for a single periodic orbit may be computed explicitly.

In \S \ref{section:Markov_partition}, we recall results on Markov partition from \cite{bowen_markov}. In particular, we explain how one can get rid of the errors in the counting of periodic orbits due to the lack of injectivity of symbolic coding. We use then these results to reduce the proof of Theorems \ref{theorem:bound_resonances} and \ref{theorem:bound_determinant} to a bound on dynamical determinants associated to systems of open hyperbolic maps (Lemma \ref{lemma:individual_estimates}).

In \S \ref{section:dynamical_determinant}, we establish Lemma \ref{lemma:individual_estimates}, ending the proof of Theorems \ref{theorem:bound_resonances} and \ref{theorem:bound_determinant}. This is done by designing spaces of ultradistributions adapted to families of Gevrey open hyperbolic maps (similar to the spaces constructed for Gevrey Anosov maps in \cite[Appendix B]{jezequel_distribution}).

Finally, we isolated some technical results on entire functions in Appendix~\ref{appendix:bound_entire_functions}.

\subsection*{Acknowledgements}

The author benefits from the support of the French government “Investissements d’Avenir” program integrated to France 2030, bearing the following reference ANR-11-LABX-0020-01.

\section{Preliminaries}\label{section:preliminaries}

We begin with some preliminary discussions of hyperbolic flows (\S \ref{subsection:hyperbolic_flows}) and Gevrey regularity (\S \ref{subsection:Gevrey}). After recalling all the necessary definitions, we will be able in \S \ref{subsection:results} to state more detailed versions of Theorems \ref{theorem:resonances_simple} and \ref{theorem:determinant_simple}: Theorems \ref{theorem:bound_resonances} and \ref{theorem:bound_determinant}.

\subsection{Hyperbolic flows and dynamical determinant}\label{subsection:hyperbolic_flows}

Let $M$ be a smooth ($C^\infty$) compact manifold and $\phi = (\phi_t)_{t \in \mathbb{R}}$ be a smooth flow with generator $X$. We endow $M$ with any smooth Riemannian metric. Let $K \subseteq M$ be a closed $\phi$-invariant set. We assume also that $\phi$ has no fixed point in $K$. Let us recall some standard definitions.

\begin{definition}[{\cite[Definition 17.4.1]{katok_hasselblatt}}]\label{definition:hyperbolic}
We say that $K$ is \emph{hyperbolic} for $\phi$ if:
\begin{itemize}
\item for every $x \in K$, the tangent space $T_x M$ splits as
\begin{equation}\label{eq:hyperbolic_splitting}
T_x M = E_{\un}(x) \oplus E_\st(x) \oplus \langle X(x) \rangle;
\end{equation}
\item for every $x \in K$ and $t \in \mathbb{R}$, we have
\begin{equation*}
D \phi_t(x) E_{\un}(x) = E_{\un}(\phi_t(x)) \textup{ and } D \phi_t(x) E_{\st}(x)= E_{\st}(\phi_t(x));
\end{equation*}
\item there are constants $C,\lambda > 0$ such that for every $x \in K, v_\un \in E_\un(x), v_\st \in E_\st(x)$ and $t \geq 0$ we have
\begin{equation*}
|D \phi_ t (x) \cdot v_\st | \leq C e^{- \lambda t} |v_\st| \textup{ and } |D \phi_{-t}(x) \cdot v_\un | \leq C e^{- \lambda t} |v_\un|.
\end{equation*}
\end{itemize}
\end{definition}

\begin{definition}[{\cite[(1.1)]{bowen_markov}}]
We say that $K$ is \emph{a basic set} for $\phi$ if:
\begin{itemize}
\item $K$ contains no fixed point of $\phi$ and is hyperbolic for $\phi$;
\item periodic orbits of $\phi$ contained in $K$ are dense in $K$;
\item the restriction of $\phi$ to $K$ is transitive;
\item there is an open set $U$ in $M$ such that $K = \bigcap_{t \in \mathbb{R}} \phi_t(U)$.
\end{itemize}
\end{definition}

Until the end of this subsection, we assume that $K$ is a basic set for $\phi$. Consider $p : E \to M$ a smooth vector bundle over $M$ and $\Phi$ a smooth lift of $\phi$ on $E$. I.e. $\Phi = (\Phi_t)_{t \in \mathbb{R}}$ is a flow on $E$ such that for every $t \in \mathbb{R}$ we have $p \circ \Phi_t = \phi_t \circ p$, and for every $x \in \mathbb{R}$ the map induced by $\Phi_t$ from $E_x \to E_{\phi_t(x)}$ is linear. From these data, one can define a dynamical determinant by the formula
\begin{equation}\label{eq:definition_dynamical_determinant}
d(z) = \exp\left( - \sum_{\gamma \in \mathfrak{P}} \frac{T_\gamma^{\#}}{T_\gamma}\frac{\tr(\Phi_{\gamma})}{|\det(I - \mathcal{P}_\gamma)|} e^{- z T_\gamma} \right)
\end{equation}
which makes sense for $z \in \mathbb{C}$ with $\re z \gg 1$. In this formula, $\mathfrak{P}$ denotes the set of periodic orbits in $K$ for $\phi$ and if $\gamma \in \mathfrak{P}$, we let
\begin{itemize}
\item $T_\gamma$ be the length of $\gamma$;
\item $T_\gamma^{\#}$ be the primitive length of $\gamma$, that is the length of the shortest periodic orbit with the same image as $\gamma$;
\item $\Phi_\gamma$ be the map induced by $\Phi_{T_\gamma}$ on $E_x = p^{-1}(\set{x})$ for any $x$ on $\gamma$;
\item $\mathcal{P}_\gamma$ be the linearized Poincaré map, that is the map induced by $D \phi_{T_\gamma}(x)$ on $E_{\un}(x) \oplus E_{\st}(x)$ for any $x$ on $\gamma$.
\end{itemize}
Notice that the linear maps $\mathcal{P}_\gamma$ and $\Phi_\gamma$ depend on the choice of a point $x$ on $\gamma$, but their conjugacy classes do not, so that \eqref{eq:definition_dynamical_determinant} makes sense. Let us recall \cite{bowen_periodic_orbits} that, for every $\epsilon > 0$, the number of periodic orbits for $\phi$ in $K$ of length less than $T$ is $\mathcal{O}(e^{(h(\phi)+\epsilon) T})$ as $T$ goes to $+ \infty$. Here, $h(\phi)$ denotes the topological entropy of $\phi$. It follows from this estimate on the number of periodic orbits that the formula \eqref{eq:definition_dynamical_determinant} indeed makes sense when $\re z \gg 1$.

The most important fact about the dynamical determinant $d(z)$ is that it has a holomorphic continuation to $\mathbb{C}$ whose zeroes are the Ruelle resonances\footnote{In this paper, we will use this fact as a definition for Ruelle resonances. For more intrinsic definitions and basic properties, one may refer to the textbooks \cite{baladi_book,lefeuvre_book}. The first reference deals with the discrete-time case. Notice that there is a notion of multiplicity for resonances. With our definition, the multiplicity of a resonance is just its multiplicity as a zero of the dynamical determinant. We will always count resonances according to multiplicity.} associated to $\Phi$. This is proven in this context in \cite[Theorem 4]{dyatlov_guillarmou_open_systems} but see also \cite{glp_zeta,dyatlov_zworski_zeta} for the Anosov flows case.

The topological properties of hyperbolic flows have been extensively studied. We will need the following result, known as the expansivity properties. For a proof, see for instance \cite[Proposition (1.6)]{bowen_periodic_orbits} and \cite[Theorem 3]{bowen_walters_oneparameter}.

\begin{proposition}\label{proposition:expansivity}
For every $\epsilon > 0$, there is a $\nu > 0$ such that if $(u_i)_{i \in \mathbb{Z}}$ and $(t_i)_{i \in  \mathbb{Z}}$ are sequences of real numbers such that $t_{i} \underset{i \to + \infty}{\to} + \infty$, $t_{i} \underset{i \to - \infty}{\to} -\infty$ and $0 < t_{i+1} - t_i < \nu$ and $|u_{i+1} - u_i| < \nu$ for every $i \in \mathbb{Z}$, and $x,y \in K$ are such that $d(\phi_{t_i}(x),\phi_{u_i}(y)) \leq \nu$ for every $i \in \mathbb{Z}$, then there is $t \in [-\epsilon,\epsilon]$ such that $y = \phi_t(x)$.
\end{proposition}

\subsection{Gevrey regularity}\label{subsection:Gevrey}

Let us fix $s > 1$ and recall some basic facts about $s$-Gevrey or, from now on, $G^s$ regularity.

\begin{definition}
Let $n \geq 1$ and $U$ be an open subset of $\mathbb{R}^n$. We say that a function $f : U \to \mathbb{C}$ is $G^s$ if $f$ is $C^\infty$ and for every compact subset $K$ of $U$ there are constants $C, R > 0$ such that for every $\alpha \in \mathbb{N}^n$ and $x \in K$, we have
\begin{equation*}
|\partial^\alpha f(x)| \leq C R^{|\alpha|} \alpha!^s.
\end{equation*}
As usual, for $\alpha = (\alpha_1,\dots, \alpha_n) \in \mathbb{N}^n$, we define $\alpha!$ as $\alpha_1! \dots \alpha_n!$.
\end{definition}

The class of $G^s$ functions shares many properties with the class of $C^\infty$ functions: it is closed under composition and the inverse function theorem (and thus the implicit function theorem and its standard consequences) holds in the $G^s$ category. Consequently, basic differential geometry translates naturally to the $G^s$ setting: a $G^s$ manifold is just a second countable Hausdorff topological space endowed with a maximal $G^s$ atlas (i.e. an atlas with $G^s$ change of charts).

Notice that since the class of $G^s$ functions is closed by differentiation, the tangent bundle $TM$ of a $G^s$ manifold $M$ has a natural structure of $G^s$ vector bundle. Hence, it makes sense to say that a vector field $X$ on $M$ is $G^s$. Moreover, since the class of $G^s$ function is closed under solving ODEs, if $(\phi_t)_{t \in \mathbb{R}}$ denotes the flow of $X$, then $X$ is $G^s$ if and only if the map $(t,x) \mapsto \phi_t(x)$ is. 

For general results on Gevrey regularity (and other Denjoy--Carleman classes of regularity), the reader may refer to \cite{convenient_setting} and references therein.

Let us now discuss some more technical facts about Gevrey regularity, the goal being to prove Lemma \ref{lemma:gevreyification} below. We will use this result in \S \ref{section:Markov_partition} to justify the existence of a Markov partition whose rectangles are contained in $G^s$ disks. Let us start with an approximation lemma.

\begin{lemma}\label{lemma:approximation_gevrey}
Let $M$ and $N$ be $G^s$ manifolds. Then $G^s$ maps from $M$ to $N$ are dense in the set of of $C^1$ maps from $M$ to $N$ endowed with the strong $C^1$ topology.
\end{lemma}

\begin{proof}
One can just notice that the proof in the $C^\infty$ case also applies in $G^s$ regularity, once we know that there are $G^s$ bump functions. One can for instance look at \cite[Chapter 2, \S 2]{hirsch_differential_topology}, see in particular the proof of Theorem 2.6 there. In this proof, if all the bump functions, charts and convolution kernels are chosen to be $G^s$, then the resulting approximation functions will also be $G^s$. 
\end{proof}

We say that a subset $D$ of an $n$-dimensional $C^1$ manifold $M$ is a $C^1$ disk if there is a $C^1$ embedding of the open unit ball in $\mathbb{R}^{n-1}$ in $M$ whose image is $D$. The definition of $G^s$ disk is obtained by replacing all the occurrences of $C^1$ by $G^s$ in this definition. We will use the following fact:

\begin{lemma}\label{lemma:gsdisk}
Let $M$ be a $G^s$ manifold and $D \subseteq M$ a $C^1$ disk. Assume that $D$ is a $G^s$ submanifold of $M$, then $D$ is a $G^s$ disk.
\end{lemma}

\begin{proof}
Let $n$ denote the dimension of $M$ and $B$ be the unit ball in $\mathbb{R}^{n-1}$. By assumption there is a $C^1$ embedding $\kappa : B \to M$ whose image is $D$. In particular, $\kappa$ induces a $C^1$ diffeomorphism between $B$ and $D$. According to Lemma \ref{lemma:approximation_gevrey}, we may approximate $\kappa$ by a $G^s$ map $\tilde{\kappa}$ in the space of $C^1$ map from $B$ to $D$ endowed with the strong topology. Since the set of $C^1$ diffeomorphisms is open for this topology \cite[Chapter 2, Theorem 1.6]{hirsch_differential_topology}, we may assume that $\tilde{\kappa}$ is a $C^1$ diffeomorphism between $B$ and $D$. Hence, $\tilde{\kappa}$ is a $G^s$ diffeomorphism between $B$ and $D$ (because the inverse function theorem holds in $G^s$ regularity). It follows that $\tilde{\kappa}$ defines a $G^s$ embedding of $B$ in $M$ whose image is $D$.
\end{proof}

The following result will be used in \S \ref{section:Markov_partition} to upgrade the Markov partitions given by \cite{bowen_markov}.

\begin{lemma}\label{lemma:gevreyification}
Let $s > 1$ and $M$ be a $G^s$ manifold. Let $D \subseteq M$ be a $C^1$ disk and $K \subseteq D$ be compact. Let $X$ be a $C^1$ vector field on $M$, transverse to $D$, and $\phi$ be the associated flow. Let $\epsilon > 0$. Then, there is a $C^1$ disk $D' \subseteq D$ containing $K$ and a $C^1$ function $\tau : D' \to (-\epsilon,\epsilon)$ such that:
\begin{itemize}
\item for every $x \in D'$, the flow $\phi$ is defined at time $\tau(x)$;
\item the set $\set{ \phi_{\tau(x)}(x) : x \in D'}$ is a $G^s$ disk.
\end{itemize}
\end{lemma}

\begin{proof}
Let $D_0 \subseteq D$ be a $C^1$ disk such that $\overline{D}_0$ is compact and contained in $D$ and $K \subseteq D_0$. Since $X$ is transverse to $D$, there is $\epsilon_0$ such that for every $x \in D_0$ the integral curve for $X$ starting at $x$ is defined on $[- \epsilon_0,\epsilon_0]$ and the map
\begin{equation*}
\begin{array}{ccccc}
\Phi & : & D_0 \times ( - \epsilon_0,\epsilon_0) & \to & M \\
 & & (x,t) & \mapsto & \phi_t(x)
\end{array}
\end{equation*}
is a $C^1$ diffeomorphism on its image. Up to making $\epsilon_0$ smaller, we may assume that it is smaller than $\epsilon/2$. Let $U$ denote the image of $\Phi$. This is an open set of $M$, and there is a $C^1$ function $\tau_0 : U \to (-\epsilon_0,\epsilon_0)$ such that $\phi_{-\tau_0(x)}(x) \in D_0$ for every $x \in U$ (just take the last coordinate of $\Phi^{-1}(x)$).

Let then $D_1$ be a $C^1$ disk such that $\overline{D}_1 \subseteq D_0$ and containing $K$. Let $V = \Phi(D_1 \times ( - \epsilon_0/2,\epsilon_0/2))$. According to Lemma \ref{lemma:approximation_gevrey}, we may approximate $\tau_0$ on $V$ by a $G^s$ function $\tau_1 : V \to (-\epsilon_0,\epsilon_0)$ such that for every $x \in V$ we have $|\tau_1(x) - \tau_0(x)| \leq \delta$ and $|1 - X \tau_1(x)| \leq \delta$ where $\delta = \epsilon_0/(3 + \epsilon_0)$.

Since $\delta /(1 - \delta) < \epsilon_0/2$, we find that for every $x \in D_1$, there is a unique $\tau(x) \in ( - \epsilon_0/2,\epsilon_0/2)$ such that $\tau_1(\phi_{\tau(x)}(x)) = 0$. It follows from the Implicit Function Theorem that the function $\tau$ constructed this way is $C^1$. Moreover, it follows from the construction of $\tau$ that
\begin{equation*}
\set{ \phi_{\tau(x)}(x) : x \in D_1} = \set{ x \in V : \tau_1(x) = 0}.
\end{equation*}
Since $X \tau_1$ does not vanish, $\tau_1$ is a $G^s$ submersion on $V$, and thus the set $\set{ x \in V : \tau_1(x) = 0}$ is a $G^s$ manifold, hence a $G^s$ disk according to Lemma~\ref{lemma:gsdisk}.
\end{proof}

\subsection{Detailed results}\label{subsection:results}

With the context introduced in \S \ref{subsection:hyperbolic_flows} and \S \ref{subsection:Gevrey}, we can now state more general versions of Theorems \ref{theorem:resonances_simple} and \ref{theorem:determinant_simple}. We start with a bound on the number of resonances, recalling that resonances may be defined as the zeroes the (holomorphic continuation of the) dynamical determinant \eqref{eq:definition_dynamical_determinant}.

\begin{theorem}\label{theorem:bound_resonances}
Let $s > 1$. Assume that $\phi$ is a $G^s$ flow on a compact $G^s$ manifold $M$ of dimension $n$ with a basic set $K \subseteq M$. Let $E \to M$ be a $G^s$ vector bundle over $M$ and $\Phi : E \to E$ a $G^s$ lift of $\phi$ on $E$. For $r > 0$, let $N(r)$ be the number of associated resonances of modulus less than $r$. There is a constant $C > 0$ such that for every $r \geq 1$ we have
\begin{equation*}
N(r) \leq C r^{1 + (n-1)s}.
\end{equation*}
\end{theorem}

We have a related result for the dynamical determinant :

\begin{theorem}\label{theorem:bound_determinant}
Let $s > 1$. Assume that $\phi$ is a $G^s$ flow on a compact $G^s$ manifold $M$ of dimension $n$ with a basic set $K \subseteq M$. Let $E \to M$ be a $G^s$ vector bundle over $M$ and $\Phi : E \to E$ a $G^s$ lift of $\phi$ on $E$. Let $d(z)$ be the dynamical determinant defined as the holomorphic continuation to $\mathbb{C}$ of \eqref{eq:definition_dynamical_determinant}. There is a constant $C > 0$ such that for every $z \in \mathbb{C}$ we have
\begin{equation*}
|d(z)| \leq \begin{cases} C \exp \left(C |z|^{1 + (n-1)s}\right) & \textup{ if } 1 + (n-1)s \textup{ is not an integer,} \\
                          C \exp \left(C|z|^{1 + (n-1)s} \log(1 + |z|)\right) & \textup{ if } 1 + (n-1)s \textup{ is an integer.} \end{cases}
\end{equation*}
\end{theorem}

Theorems \ref{theorem:bound_resonances} and \ref{theorem:bound_determinant} improve \cite[Theorem 3.1 and Proposition 3.2]{asterisque} in the case $s > 1$ by replacing the exponent $ns$ from this reference by $1 + (n-1)s$. The difficulty with integral exponent in Theorem \ref{theorem:bound_determinant} does not appear in \cite{asterisque} though. We suspect that the fact that the bound is slightly worst when $1 + (n-1)s$ is merely an artifact of the proof. The difficulty appears precisely in Lemma \ref{lemma:bound_entire_function}. On a technical level it is related to the fact that, altough for every $\alpha > 1$ we have
\begin{equation*}
\sum_{k = 1}^n \frac{1}{k^\alpha} \underset{n \to + \infty}{\sim} \frac{1}{1- \alpha} \frac{1}{n^{\alpha - 1}},
\end{equation*}
for $\alpha = 0$ we have
\begin{equation*}
\sum_{k = 1}^n \frac{1}{k} \underset{n \to + \infty}{\sim} \log(n),
\end{equation*}
and the right-hand side in this estimate is not exactly a constant term.

\section{Case of a single periodic orbit}\label{section:periodic_orbit}

In this section, we prove Theorems \ref{theorem:bound_resonances} and \ref{theorem:bound_determinant} in the case of a basic set which is reduced to a single periodic orbit. Indeed, the results on Markov partitions that we will use in the proof of Theorems \ref{theorem:bound_resonances} and \ref{theorem:bound_determinant} below do not apply in this situation. Fortunately, this case may be dealt with explicitly. Actually, the Gevrey regularity assumption is not needed in this specific situation.

In this section, $\phi = (\phi_t)_{t \in \mathbb{R}}$ denotes a smooth flow on a smooth manifold $M$, and $K \subseteq M$ is a basic set for $\phi$ that is reduced to a single periodic orbit. I.e. there is $x_0 \in K$ such that $K= \set{\phi_t(x_0) : t \in \mathbb{R}}$. It implies that $x_0$ is a periodic point, and we will denote by $t_0$ its minimal period. As above, let $E \to M$ be a smooth vector bundle and $\Phi :E \to E$ a smooth lift of $\phi$. 

If $\gamma_0$ denotes the orbit of $\phi$ of length $t_0$ passing through $x_0$, then the periodic orbits of $\phi$ are\footnote{If $\gamma$ is a periodic orbit for a flow $\phi$ and $p$ a positive integer, we denote by $\gamma^p$ the periodic orbit of $\phi$ obtained by going $p$ times along $\gamma$.} $\mathfrak{P} = \set{\gamma_0^p : p \in \mathbb{N}^*}$. Consequently for $z \in \mathbb{C}$ with $\re z \gg 1$, we have
\begin{equation*}
\sum_{\gamma \in \mathfrak{P}} \frac{T_\gamma^{\#}}{T_\gamma}\frac{\tr(\Phi_{\gamma})}{|\det(I - \mathcal{P}_\gamma)|} e^{- z T_\gamma} = \sum_{p \geq 1} \frac{1}{p} \frac{\tr(\Phi_{\gamma_0}^p)}{|\det(I - \mathcal{P}_{\gamma_0}^p)|} e^{- z pt_0}.
\end{equation*} 
Let $\lambda_1,\dots,\lambda_{d_u}$ and $\mu_1,\dots,\mu_{d_s}$ denote the eigenvalues of $\mathcal{P}_{\gamma_0}$ respectively of modulus larger than one and smaller than one. Let then $q_-$ denote the number of eigenvalues of $\mathcal{P}_{\gamma_0}$ in $(-\infty,-1)$. Then, for every $p \geq 1$, we have
\begin{equation*}
\begin{split}
|\det(I - \mathcal{P}_{\gamma_0}^p)|^{-1} & = (-1)^{p q_-} \left( \prod_{j = 1}^{d_u} (\lambda_j^p - 1) \prod_{k = 1}^{d_s}(1 - \mu_k^p) \right)^{-1} \\
         & = (-1)^{p q_-} \prod_{j= 1}^{d_u} \lambda_j^{-p} \left( \prod_{j = 1}^{d_u} (1 - \lambda_j^{-p}) \prod_{k = 1}^{d_s}(1 - \mu_k^p) \right)^{-1} \\
         & = \sum_{\substack{k_1,\dots,k_{d_u} \geq 1 \\ \ell_1,\dots,\ell_{d_s} \geq 0}} \left( (-1)^{q_-} \prod_{j= 1}^{d_u} \lambda_j^{-k_j} \prod_{k = 1}^{d_s} \mu_k^{\ell_k} \right)^p
\end{split}
\end{equation*}
Let then $\mathfrak{R}$ be the set of numbers of the form
\begin{equation}\label{eq:explicit_resonances}
\zeta (-1)^{q_-} \prod_{j= 1}^{d_u} \lambda_j^{-k_j} \prod_{k = 1}^{d_s} \mu_k^{\ell_k} 
\end{equation}
with $k_1,\dots,k_{d_u} \geq 1, \ell_1,\dots,\ell_{d_s} \geq 0$ and $\zeta$ an eigenvalue of $\Phi_{\gamma_0}$. The elements of $\mathfrak{R}$ have a multiplicity which is the number of ways to write them in the form \eqref{eq:explicit_resonances}, taking into account the multiplicities of the eigenvalues of $\Phi_{\gamma_0}$. This notation and the computation above yields for $\re z \gg 1$:
\begin{equation*}
\begin{split}
\sum_{\gamma \in \mathfrak{P}} \frac{T_\gamma^{\#}}{T_\gamma}\frac{\tr(\Phi_{\gamma})}{|\det(I - \mathcal{P}_\gamma)|} e^{- z T_\gamma} & = \sum_{\rho \in \mathfrak{R}} \sum_{p \geq 1} \frac{1}{p} (\rho e^{-z t_0})^p \\
        & = - \sum_{\rho \in \mathfrak{R}} \ln(1 - \rho e^{-z t_0}),
\end{split}
\end{equation*} 
where in the summation over $\mathfrak{R}$, the terms are repeated according to their multiplicities. It follows that for $\re z \gg 1$ we have
\begin{equation}\label{eq:periodic_determinant}
d(z) = \prod_{\rho \in  \mathfrak{R}} (1 - \rho e^{-z t_0}).
\end{equation}
From the definition of $\mathfrak{R}$, we find that the numbers of element of $\mathfrak{R}$ (counted with multiplicities) larger than some small $r > 0$ is a $\mathcal{O}(|\ln r|^{n-1})$. It follows from Lemma \ref{lemma:simple_bound} that the infinite product in \eqref{eq:periodic_determinant} actually converges for every $z \in \mathbb{C}$. Hence, we retrieve that $d(z)$ has a holomorphic extension to $\mathbb{C}$. Moreover, Lemma \ref{lemma:simple_bound} gives the bound
\begin{equation*}
|d(z)| \leq C \exp(C |\re z|^n)
\end{equation*}
for some $C > 0$ and every $z \in \mathbb{C}$, which is better than the bound from Theorem \ref{theorem:bound_determinant}. The bound on the number of resonances follows from Lemma \ref{lemma:bound_entire_function} (with ``$g = 0$'') or just by noticing that the resonances are, up to a factor $t_0^{-1}$, the logarithms of the elements of $\mathfrak{R}$.

\section{Markov partitions}\label{section:Markov_partition}

In this section, we discuss Markov partitions. They will play a key role in the proofs of Theorems \ref{theorem:bound_resonances} and \ref{theorem:bound_determinant}. We start in \S \ref{subsection:definition_Markov} with very general notions, but with a small twist that we will need later to deal with Gevrey regularity. We recall then in \S \ref{subsection:symbolic_coding} how to use a Markov partition to construct a symbolic flow semi-conjugated to a hyperbolic flow. It is well-known that this semi-conjugacy introduce errors when counting periodic orbits. There is however an argument of Bowen to fix this issue. We recall this argument in \S \ref{subsection:exact_counting} and use it to reduce the proofs of Theorems \ref{theorem:bound_resonances} and \ref{theorem:bound_determinant} to a statement about some symbolic dynamical systems (Lemma \ref{lemma:individual_estimates}) that we will prove in \S \ref{section:dynamical_determinant}.

Our exposition of Markov partitions and related topics is mostly based on \cite{bowen_markov}. Other possible references include \cite{chernov_markov_approximation} and \cite{fisher_hasselblatt}.

\subsection{Definition of Markov partitions}\label{subsection:definition_Markov}

From now on, let $\phi = (\phi_t)_{t \in \mathbb{R}}$ be a $G^s$ flow on a $G^s$ manifold $M$, and $K \subseteq M$ a basic set for $\phi$. We assume that $K$ is not reduced to a single periodic orbit (this case has been dealt with in \S \ref{section:periodic_orbit}). Let $d$ denote the distance associated to the smooth Riemannian metric on $M$. For $x \in K$ and $\epsilon_{0} > 0$, let us define the local strong stable and unstable manifold of $x$ as
\begin{equation}\label{eq:stable_manifold}
\begin{split}
W_{\epsilon_0}^{\st}(x) = \Big\{ y \in M : d(\phi_t(x),\phi_t(y)) & \underset{t \to + \infty}{\to} 0 \\ & \textup{ and } \forall t \in \mathbb{R}_+: d(\phi_t(x),\phi_t(y)) < \epsilon_0\Big\}
\end{split}
\end{equation}
and
\begin{equation}\label{eq:unstable_manifold}
\begin{split}
W_{\epsilon_0}^{\un}(x) = \Big\{ y \in M : d(\phi_{-t}(x),\phi_{-t}(y)) & \underset{t \to + \infty}{\to} 0 \\ & \textup{ and } \forall t \in \mathbb{R}_+: d(\phi_{-t}(x),\phi_{-t}(y)) < \epsilon_0 \Big\}.
\end{split}
\end{equation}
If $\epsilon_0 > 0$ is small enough, then these sets are indeed smooth submanifolds of $M$, tangent respectively to $E_{\st}(x)$ and $E_{\un}(x)$ at $x$. However, the dependence of these manifolds on $x$ is not smooth in general.

From now on, we fix $\epsilon_0 > 0$ small enough so that \eqref{eq:stable_manifold} and \eqref{eq:unstable_manifold} are smooth manifolds and the following property \cite[Proposition 6.2.2 and Theorem 6.2.7]{fisher_hasselblatt} holds true: there is $\delta_0 > 0$ such that for every $x,y \in K$ with $d(x,y) \leq \delta_0$ there is a unique $t \in [ - \epsilon_0,\epsilon_0]$ such that the manifolds $W_{\epsilon_0}^{\st}(\phi_t(x))$ and $W_{\epsilon_0}^{\un}(y)$ intersect. Moreover, there is a unique point of intersection $[x,y]$, and it belongs to $K$. The map $ (x,y) \mapsto [x,y]$ is continuous from $\set{(x,y) \in K^2: d(x,y) \leq \delta_0}$ to $M$.

The basic blocks for Markov partitions are so-called rectangles. If for three-dimensional Anosov flows, one may construct a Markov partition for which rectangles actually look like rectangles \cite[Theorem 9.1]{chernov_markov_approximation}, the situation is in general more complicated.

\begin{definition}\label{definition:rectangle}
Let $R$ be a closed subset of $K$. We say that $R$ is a rectangle if the following holds:
\begin{enumerate}[label=(\roman*)]
\item $R$ has diameter less than $\delta_0$;
\item there are a smooth disk $D$ containing $R$ and $\zeta > 0$ small such that the map
\begin{equation*}
\begin{array}{ccc}
D \times (- \zeta,\zeta) & \to & M \\
(x,t) & \to & \phi_t(x)
\end{array}
\end{equation*}
is a diffeomorphism on its image $U$;
\item if $(\pi,\tau): U \to D \times (- \zeta,\zeta)$ denote the inverse of the diffeomorphism from the previous item, then for every $x,y \in R$ we have $[x,y] \in U$ and $\pi([x,y]) \in R$.
\end{enumerate}
The point $\pi([x,y])$ in the last item will be denoted by $[x,y]_R$ (it does not depend on the choice of $D$ and $\zeta$). We will also denote by $R^*$ the interior of $R$ as a subset of $D$ (that does not depend on $D$ either). We say that $R$ is proper if $R = \overline{R^*}$.

If $R$ is a rectangle and $x \in R$, we define
\begin{equation*}
W_R^{\st}(x) = \set{[x,y]_R : y \in R} \textup{ and } W_R^{\un}(x) = \set{[y,x]_R : y \in R}.
\end{equation*}
\end{definition}

We can now recall the definition of a Markov partition.

\begin{definition}\label{definition:markov}
Let $\kappa > 0$. Let $(R_i)_{i \in I}$ be a finite family of proper rectangles such that every flow line for $\phi$ in $K$ intersects $\Omega \coloneqq \bigcup_{i \in I} R_i$ in positive and negative time. Let $T: \Omega \to \Omega$ be the first return map for $\phi$ and $\Omega^{*} = \cap_{m \in \mathbb{Z}} T^m( \bigcup_{i \in I} R_i^*)$. We say that $(R_i)_{i \in I}$ is a Markov partition for $\phi$ of size $\kappa > 0$ if\footnote{If $I$ is a subset of $\mathbb{R}$ and $A$ a subset of $M$, we write $\phi_I(A) = \set{\phi_t(x) : t \in I, x \in A}$.}:
\begin{enumerate}[label=(\roman*)]
\item $K = \phi_{[ - \kappa,0]}(\Omega)$;
\item for each $i \in I$, the rectangle $R_i$ is contained in a smooth disk $D_i$ of diameter less than $\kappa$ with the properties from Definition \ref{definition:rectangle};
\item if $i$ and $j$ are distinct elements of $I$ then at least one of the two sets $\phi_{[-\kappa,0]} (D_i) \cap D_j$ and $\phi_{[0,\kappa]}(D_i) \cap D_j$ is empty; \label{item:disjoint}
\item if $i,j \in I$ and $x \in R_{i,j} \coloneqq \overline{\set{ y \in \Omega^* : y \in R_i \textup{ and } T y  \in R_j}}$ then $W_{R_i}^{\st}(x) \subseteq R_{i,j}$;
\item if $i,j \in I$ and $x \in \overline{\set{ y \in \Omega^* : y \in R_i \textup{ and } T^{-1} y  \in R_j}}$ then $W_{R_i}^{\un}(x) \subseteq \overline{\set{ y \in \Omega^* : y \in R_i \textup{ and } T^{-1} y  \in R_j}}$.
\end{enumerate}
We shall say that $(R_i)_{i \in I}$ is a $G^s$ Markov partition if in addition the disks $D_i$'s may be chosen $G^s$.
\end{definition}

The existence of Markov partitions of arbitrarily small size for basic sets of hyperbolic flows has been established by Bowen \cite[Theorem (2.5)]{bowen_markov}. When the manifold $M$ is $G^s$, it is not very hard to deduce the existence of $G^s$ Markov partitions. Quite notably, the fact that the flow $\phi$ is $G^s$ or not does not play a role in this matter. To construct a $G^s$ Markov partition (of arbitrarily small size), one can just start from a standard Markov partition and use Lemma \ref{lemma:gevreyification} to replace the smooth disks by $G^s$ disks. This modification only involves arbitrarily small translation by the flow, so that the result is still a Markov partition (maybe of a slightly larger size).

\subsection{Symbolic coding}\label{subsection:symbolic_coding}

Let us now recall how one can use a Markov partition to code a hyperbolic flow using a subshift of finite type. Let us fix for this subsection a $G^s$ Markov partition $(R_i)_{i \in I}$, associated to $\phi$ and $K$, of small size $\kappa > 0$. We will only work with Markov partitions of sufficiently small size so that the results from \cite{bowen_markov} apply.

Let $\Omega = \bigcup_{i \in I} R_i$ and $T : \Omega \to \Omega$ be the first return map for $\phi$. Define the adjacency matrix $ A \in \set{0,1}^{I \times I}$ by
\begin{equation*}
A_{i,j} = \begin{cases} 1 & \textup{ if there is } x \in R_i^* \textup{ such that } T x \in R_j^*, \\ 0 & \textup{ otherwise.} \end{cases}
\end{equation*}
We let then $\Sigma_A$ be the set of bi-infinite paths in the graph whose adjacency matrix is $A$:
\begin{equation*}
\Sigma_A = \set{ (i_m)_{m \in  \mathbb{Z}} \in I^{\mathbb{Z}} : \forall m \in \mathbb{Z}, A_{i_m,i_{m+1}} = 1}.
\end{equation*}
Notice that the shift $\sigma :(i_m)_{m \in \mathbb{Z}} \mapsto (i_{m+1})_{m \in \mathbb{Z}}$ let $\Sigma_A$ invariant. We endow $\Sigma_A$ with the metric
\begin{equation*}
d((i_m)_{m \in \mathbb{Z}}, (j_m)_{m \in \mathbb{Z}}) = 2^{- m_0}, \quad m_0 = \inf \set{ m \in \mathbb{N} : i_m \neq j_m \textup{ or } i_{-m} \neq j_{-m}},
\end{equation*}
with the convenction that $2^{- \infty} = 0$.

Notice that if $i,j \in I$ are such that $A_{i,j} = 1$ then the set $R_{i,j}$ from Definition \ref{definition:markov} is non-empty. Moreover, the restriction of the first return map $T$ to $\set{x \in R_i^* : T x \in R_j^*} \subseteq R_{i,j}$ extends to a continuous map $T_{i,j} : R_{i,j} \to R_j$. In addition, it follows from the implicit function theorem that $T_{i,j}$ is of the form $x \mapsto \phi_{t_{i,j}(x)}(x)$ where $t_{i,j}$ is the restriction of a $G^s$ function on $D_i$ to $R_{i,j}$ (this is where we use the fact that $D_i$ is a $G^s$ disk).

There is a surjective H\"older-continuous map $\pi_A : \Sigma_A \to \Omega$ such that for every $\mathbf{i} = (i_m)_{m \in \mathbb{Z}} \in \Sigma_A$ we have
\begin{equation}
\pi_A(\mathbf{i}) \in R_{i_0,i_1} \textup{ and } T_{i_0,i_1}(\pi(\mathbf{i})) = \pi(\sigma(\mathbf{i})).
\end{equation}
See \cite[\S 2]{bowen_markov}, in particular Lemma (2.2). Hence, letting $\tau$ be the function defined on $\Sigma_A$ by $\tau(\mathbf{i}) = t_{i_0,i_1}(\pi_A(\mathbf{i}))$, one finds that $\phi$ is semi-conjugated \cite[Definition 1.3.1]{fisher_hasselblatt} to the suspension flow of $(\Sigma_A,\sigma)$ with roof function $\tau$. The issue is that this semi-conjugacy may induce errors when counting periodic orbits, but this can be fixed as we will see in \S \ref{subsection:exact_counting}. 

\begin{remark}\label{remark:suspension_flow}
As the notion just appeared, let us recall the definition of a suspension flow. Let $X$ be a set and $T : X \to X$ a bijection. Let $\tau :X \to \mathbb{R}_+^*$ be a function. The suspension of $(X,T)$ with roof function $\tau$ is the quotient $\widetilde{X}$ of $X \times \mathbb{R}$ under the action of $\mathbb{Z}$ given by $1 \cdot (x,t) = (Tx , t - \tau(x))$.

Notice that the translation flow on $X \times \mathbb{R}$ given by $\psi_t(x,t') = (x,t+t')$ for $x \in X, t,t' \in \mathbb{R}$ commutes with the action of $\mathbb{Z}$ on $X \times \mathbb{R}$. Hence, $\psi$ induces a flow $\tilde{\psi}$ on $\widetilde{X}$, that we call the suspension flow of $(X,T)$ with roof function $\tau$.

We are particularly interested in the periodic orbits for $\tilde{\psi}$. If $m \geq 1$, a periodic orbit of length $m$ for $T$ is a set of the form $\set{x,T x,\dots,T^{m-1} x}$ with $x \in X$ such that $T^m x = x$. If $\set{x,\dots,T^{m-1} x}$ is a periodic orbit of length $m$ for $T$ then the orbit of $\tilde{\psi}$ of length $\sum_{k = 0}^{m-1} \tau(T^k x)$ passing through the projection of $(x,0)$ to $\widetilde{X}$ is periodic. Moreover, this periodic orbit is primitive if and only if $m$ is the minimal period of $T$ (i.e. if $T^k x \neq x$ for $k = 1,\dots,m-1$). This construction describes exactly all the periodic orbits for $\tilde{\psi}$.

Notice that with our definitions if $x \in X$ is a periodic point for $T$ with minimal period $m_0$, then $\set{x,T x,\dots, T^{m_0 - 1}x}$ is a periodic of length $m$ for every multiple $m$ of $m_0$. The associated orbit for $\tilde{\psi}$ depends on the choice of $m$ (but they are all multiples of the primitive periodic orbits obtained by taking $m = m_0$).
\end{remark}

We want now to upgrade our Markov partition in prevision for the proof of Lemma \ref{lemma:value_trace}. Similar conditions are implicitly imposed in \cite{bowen_markov}, we detail this property in order to a lift a possible ambiguity.

\begin{lemma}\label{lemma:better_partition}
Let $T_0$ denote the length of the shortest periodic orbit in $K$ for $\phi$. Let $\kappa_0 > 0$. There is a $G^s$ Markov partition $(R_i)_{i \in I}$ for $\phi$ of size $\kappa \in (0,\kappa_0)$ such that the disks $(D_i)_{i \in I}$ in Definition \ref{definition:markov} may be chosen such that for every $i \in I, x \in D_i$ and $t \in [-T_0/2,T_0/2]$ we have $\phi_t(x) \notin D_i$.
\end{lemma}

\begin{proof}
Start with a $G^s$ Markov partition $(R_i)_{i \in I}$ of size $\kappa_1 \in (0,\kappa_0)$ (whose existence is guaranteed by \cite[Theorem 2.5]{bowen_markov} and Lemma \ref{lemma:gevreyification}). Define the adjacency matrix $A$, the subshift $\Sigma_A$ and the map $\pi_A$ as above. For each $N \geq 1$, let $I_N$ denote the set of words $\omega_{-N} \dots \omega_0 \dots \omega_{N}$ of length $2N+1$ such that $A_{\omega_j,\omega_{j+1}} =1$ for $j = -N,\dots,N-1$. For $\omega \in I_N$ define then the sets
\begin{equation*}
\mathcal{R}_{\omega} = \set{(x_j)_{j \in \mathbb{Z}} \in \Sigma_A : x_j = \omega_j \textup{ for } j = -N,\dots,N} \textup{ and } \widetilde{R}_\omega = \pi_A(\mathcal{R}_\omega).
\end{equation*}
Now, notice that for every $i \in I$, the compact rectangle $R_i$ may be covered by a finite number of disks $D \subseteq D_i$ such that $\phi_t(D) \cap D = \emptyset$ for every $t \in [-3 T_0/4,3 T_0/4]$: we deal with small $t$'s by transversality of $D_i$ and the flow and with large $t$ by the definition of $T_0$ and continuity. Now, since $\pi_A$ is H\"older continuous, for $N$ large enough all the rectangles $\widetilde{R}_\omega$ for $\omega \in I_N$ are contained in such a disk. Hence, we would like to use $(\widetilde{R}_\omega)_{\omega \in I_N}$ as a new Markov partition, but this family of rectangles may miss property \ref{item:disjoint} in Definition \ref{definition:rectangle}. We fix this issue by translating these rectangles along the flow (for very small distinct times), and we get a Markov partition of size slightly larger than $\kappa_1$ with the required property.
\end{proof}

\begin{remark}\label{remark:really_small}
From now on, we fix a $G^s$ Markov partition $(R_i)_{i \in I}$ of small size $\kappa >0$ which in addition satisfies Lemma \ref{lemma:better_partition}, with disks $(D_i)_{i \in I}$. We will in addition assume that $2 \kappa > 0$ is smaller than the $\nu$ given by Proposition \ref{proposition:expansivity} applied with $\epsilon = T_0/2$. All the discussion from \S \ref{subsection:symbolic_coding} apply to this Markov partition, and we will use the notations introduced there.
\end{remark}

\subsection{Exact counting of periodic orbits}\label{subsection:exact_counting}

We explain now a method due to Bowen \cite[\S 5]{bowen_markov}, based on previous work by Manning \cite{manning_zeta} in the discrete-time case, to count periodic orbits of a hyperbolic flow without error, using a Markov partition. To do so, we study a family of subshifts of finite type instead of only one. We keep using the notations introduced in the previous subsections.

If $i \in I$, define 
\begin{equation*}
R_i^+ = \set{\phi_t(\pi(\mathbf{i})): \mathbf{i} = (i_m)_{m \in \mathbb{Z}} \in \Sigma_A, i_0 = i, t \in [0, \tau(\mathbf{i})]}.
\end{equation*}
If $\mathbf{k} = (k_1,\dots,k_\ell)$ is a $\ell$-uple of positive integers for some $\ell$, define
\begin{equation*}
\begin{split}
Q_{\mathbf{k}} = \Big\{(V_1,\dots,V_\ell) \in \mathcal{P}(I)^{\ell} : V_1,\dots,& V_\ell \textup{ pairwise disjoint} \\ & \textup{and } \# V_j = k_j \textup{ for } j = 1,\dots,\ell \Big\},
\end{split}
\end{equation*}
where $\mathcal{P}(I)$ denote the power set of $I$. Notice that the set $Q_{\mathbf{k}}$ is empty unless $\sum_{j = 1}^\ell k_j \leq \# I$. If $\mathbf{V} = (V_1,\dots,V_\ell) \in Q_{\mathbf{k}}$, then we define $\Un(\mathbf{V}) = \bigcup_{j = 1}^\ell V_j$. We define then
\begin{equation*}
\begin{split}
I_{\mathbf{k}} = \Bigg\{(\mathbf{V},i) \in Q_{\mathbf{k}} \times I : i & \in \Un(\mathbf{V}) \textup{ and there exist } x \in R_i \\ & \textup{ and } \epsilon > 0 \textup{ such that } \phi_{[0,\epsilon]}\set{x} \subseteq \bigcap_{j \in \Un(\mathbf{V})} R_j^+\Bigg\}. 
\end{split}
\end{equation*}
With $I_{\mathbf{k}}$ comes an adjacency matrix $A_{\mathbf{k}} \in \set{0,1}^{I_{\mathbf{k}} \times I_{\mathbf{k}}}$ defined in the following way: $A_{\mathbf{k}}((\mathbf{U},i), (\mathbf{V},j) ) = 1$ if there is $j_0 \in \set{1,\dots, \ell}$ such that:
\begin{itemize}
\item $U_m = V_m$ for $m \in \set{1,\dots,\ell} \setminus \set{j_0}$;
\item there are $i' \in I$ and $W \subseteq I$ such that $A_{i',j} = 1, U_{j_0} = W \cup \set{i'}$ and $V_{j_0} = W \cup \set{j}$.
\end{itemize}
We will let $\Sigma_{\mathbf{k}}$ denote the subshift of finite type associated to the adjacency matrix $A_{\mathbf{k}}$. Notice that $\Sigma_{(1)}$ identifies with $\Sigma_A$. Moreover, there is a natural map $p_{\mathbf{k}} : \Sigma_{\mathbf{k}} \to \Sigma_A$, see \cite[Lemma (5.2)]{bowen_markov}. There is then a function $t_{\mathbf{k}} : \Sigma_{\mathbf{k}} \to (0,\kappa]$ such that for every $x \in \Sigma_{\mathbf{k}}$ we have $\phi_{t_{\mathbf{k}}(x)} \circ \pi_A \circ p_{\mathbf{k}} (x) = \pi_A \circ p_{\mathbf{k}} \circ \sigma (x)$, see \cite[Lemma 5.3]{bowen_markov}. Hence, if $\widetilde{\Sigma}_{\mathbf{k}}$ denote the suspension of $(\Sigma_{\mathbf{k}}, \sigma)$ with roof function $t_{\mathbf{k}}$ and $\psi^{\mathbf{k}}$ the associated suspension flow (see Remark \ref{remark:suspension_flow}), then there is a continuous map\footnote{The map $\rho_{\mathbf{k}}$ is not surjective in general, and thus not a semi-conjugacy.} $\rho_{\mathbf{k}} : \widetilde{\Sigma}_{\mathbf{k}} \to M$ such that $\rho_{\mathbf{k}}(\psi^{\mathbf{k}}_t(x)) = \phi_t(\rho_{\mathbf{k}}(x))$ for every $t \in \mathbb{R}$ and $x \in \widetilde{\Sigma}_{\mathbf{k}}$. Let $\mathfrak{P}_{\mathbf{k}}$ denote the set of periodic orbits of $\psi_{\mathbf{k}}$. To any $\gamma$ in $\mathfrak{P}_{\mathbf{k}}$ one may associate its image\footnote{Notice that $\gamma$ and $\rho_{\mathbf{k}}(\gamma)$ have the same length, but $\gamma$ could be primitive while $\rho_{\mathbf{k}}(\gamma)$ is not.} by $\rho_{\mathbf{k}}$ that we will denote by $\rho_{\mathbf{k}}(\gamma) \in \mathfrak{P}$. 

In the following, we will denote by $\mathcal{N}$ the set of all tuples $\mathbf{k}$ of positive integers such that $\Sigma_{\mathbf{k}}$ is non-empty (notice that the set $\mathcal{N}$ is finite). Moreover, for $\mathbf{k} \in \mathcal{N}$, we denote by $\length(\mathbf{k})$ the integer $\ell \geq 1$ such that $\mathbf{k}$ is a $\ell$-uple. We can now state a counting formula that is crucial for the proof of Theorems~ \ref{theorem:bound_resonances} and \ref{theorem:bound_determinant}.

\begin{lemma}
For every $\gamma \in \mathfrak{P}$, we have
\begin{equation}\label{eq:correction_counting}
\sum_{\mathbf{k} \in \mathcal{N}} (-1)^{\length(\mathbf{k})+1} |\rho_{\mathbf{k}}^{-1}(\set{ \gamma})| = 1.
\end{equation}
Here, $|\rho_{\mathbf{k}}^{-1}(\set{ \gamma})|$ is just the number of orbits for $\psi_{\mathbf{k}}$ that are mapped to $\gamma$ by $\rho_{\mathbf{k}}$.
\end{lemma}

\begin{proof}
We are going to use an intermediate result from the proof of \cite[Theorem (5.4)]{bowen_markov}. Let $\gamma \in \mathfrak{P}$. Write $\gamma = \gamma_0^p$ with $\gamma_0$ a primitive periodic orbit and $p \geq 1$. Let $\tau_0$ denote the length of $\gamma_0$. For each $\tau \in (0,+ \infty)$ and $\mathbf{k} \in \mathcal{N}$, let $N_\tau^{\mathbf{k}}$ denote the number of primitive periodic orbits in $\widetilde{\Sigma}_{\mathbf{k}}$ whose image by $\rho_{\mathbf{k}}$ is a multiple of $\gamma_0$ and whose period is $\tau$. From the first equation in the proof of \cite[Theorem (5.4)]{bowen_markov}, we know that for every $\tau \in (0,+ \infty)$, we have
\begin{equation}\label{eq:bowen_counting_formula}
\sum_{\mathbf{k} \in \mathcal{N}} (-1)^{\length(\mathbf{k})+1} N_{\tau}^{\mathbf{k}} = \begin{cases} 1 & \textup{ if } \tau = \tau_0, \\ 0 & \textup{ otherwise.} \end{cases}
\end{equation}
Now, let $\mathbf{k} \in \mathcal{N}$ and consider an element $\tilde{\gamma}$ of $\rho_{\mathbf{k}}^{-1}(\set{ \gamma})$. There is a primitive periodic orbit $\tilde{\gamma}_0$ for the suspension flow on $\widetilde{\Sigma}_{\mathbf{k}}$ and an integer $m$ such that $\tilde{\gamma}= \tilde{\gamma}_0^m$. Since $\tilde{\gamma}_0$ projects in $K$ to a multiple of $\gamma_0$, we see that there is an integer $d \geq 1$ such that the length of $\tilde{\gamma}_0$ is $d \tau_0$. Consequently, we have $p \tau_0 = m d \tau_0$. Hence, $d$ is a divisor of $p$. Reciprocally, if $\tilde{\gamma}_0$ is a primitive periodic orbit for the suspended flow on $\widetilde{\Sigma}_{\mathbf{k}}$ whose image by $\rho_{\mathbf{k}}$ is $\gamma_0^d$ where $d$ divides $p$, then there is a multiple of $\tilde{\gamma}_0$ in $\rho_{\mathbf{k}}^{-1}(\set{ \gamma})$. Consequently, we have
\begin{equation*}
|\rho_{\mathbf{k}}^{-1}(\set{ \gamma})| = \sum_{d|p} N_{d \tau_0}^{\mathbf{k}}.
\end{equation*}
Summing over $\mathbf{k}$, it follows that
\begin{equation*}
\begin{split}
\sum_{\mathbf{k} \in \mathcal{N}}  (-1)^{\length(\mathbf{k})+1}|\rho_{\mathbf{k}}^{-1}(\set{ \gamma})| & = \sum_{d|p} \sum_{\mathbf{k} \in \mathcal{N}} (-1)^{\length(\mathbf{k})+1} N_{d \tau_0}^{\mathbf{k}} \\ & = 1.
\end{split}
\end{equation*}
Here, we used \eqref{eq:bowen_counting_formula} to see that the term corresponding to $d = 1$ is $1$ and that all other terms are zero.
\end{proof}

In order to take advantage of \eqref{eq:correction_counting}, we associate to $\mathbf{k} \in \mathcal{N}$ a dynamical determinant
\begin{equation*}
d_{\mathbf{k}}(z) = \exp \left( - \sum_{\gamma \in \mathfrak{P}_{\mathbf{k}}} \frac{T_{\rho_{\mathbf{k}}(\gamma)}^{\#}}{T_{\rho_{\mathbf{k}}(\gamma)}}\frac{\tr(\Phi_{\rho_{\mathbf{k}}(\gamma)})}{|\det(I - \mathcal{P}_{\rho_{\mathbf{k}}(\gamma)})|} e^{- z T_{\rho_{\mathbf{k}}(\gamma)}} \right),
\end{equation*}
which is a priori only defined for $\re z \gg 1$. However, we will prove:

\begin{lemma}\label{lemma:individual_estimates}
For every $\mathbf{k} \in \mathcal{N}$, the dynamical determinant $d_{\mathbf{k}}$ has a holomorphic continuation to $\mathbb{C}$. Moreover, there is a constant $C > 0$ such that for every $z \in \mathbb{C}$ we have
\begin{equation*}
|d_{\mathbf{k}}(z)| \leq C \exp(C |z|^{1 + (n-1) s}).
\end{equation*}
\end{lemma}

Before proving this lemma, let us explain why it implies Theorems \ref{theorem:bound_resonances} and \ref{theorem:bound_determinant}.

\begin{proof}[Proof of Theorems \ref{theorem:bound_resonances} and \ref{theorem:bound_determinant}]
Let us introduce the entire functions
\begin{equation*}
f(z) = \prod_{\substack{\mathbf{k} \in \mathcal{N}\\ \length(\mathbf{k}) \textup{ odd}}} d_{\mathbf{k}}(z) \textup{ and } g(z) = \prod_{\substack{\mathbf{k} \in \mathcal{N} \\ \length(\mathbf{k}) \textup{ even}}} d_{\mathbf{k}}(z).
\end{equation*}
Notice that $f$ and $g$ satisfy the same bound as the individual $d_{\mathbf{k}}$'s (given in Lemma \ref{lemma:individual_estimates}). Moreover, it follows from \eqref{eq:correction_counting} that $d(z) = f(z)/g(z)$. Since $d(z)$ is an entire function, the result follows from Lemma \ref{lemma:bound_entire_function}.
\end{proof}

\section{Dynamical determinants for systems of open hyperbolic maps}\label{section:dynamical_determinant}

The goal of this section is to prove Lemma \ref{lemma:individual_estimates}. Let us consequently fix $\mathbf{k} \in \mathcal{N}$. For every $\alpha \in I_{\mathbf{k}}$, we define
\begin{equation*}
R_{\alpha} = \pi_A \circ p_{\mathbf{k}}(\set{(\alpha_m)_{m \in \mathbb{Z}} \in \Sigma_{\mathbf{k}} : \alpha_0 = \alpha}).
\end{equation*}
We also let $i(\alpha)$ be the element of $I$ such that there is $\mathbf{U} \in Q_{\mathbf{k}}$ such that $\alpha = (\mathbf{U}, i(\alpha))$. Notice that $R_{\alpha} \subseteq R_{i(\alpha)} \subseteq D_{i(\alpha)}$. Moreover, if $\beta \in I_{\mathbf{k}}$ is such that $A_{\mathbf{k}}(\alpha,\beta) = 1$, we also define
\begin{equation*}
R_{\alpha,\beta} = \pi_A \circ p_{\mathbf{k}}(\set{(\alpha_m)_{m \in \mathbb{Z}} \in \Sigma_{\mathbf{k}} : \alpha_0 = \alpha, \alpha_1 = \beta}).
\end{equation*}
Recall the function $t_{\mathbf{k}}$ introduced in \S \ref{subsection:exact_counting}. It follows from the implicit function theorem that there is a $G^s$ function $\tau_{\alpha,\beta} :\mathcal{U}_{\alpha,\beta} \to \mathbb{R}$ defined on an open neighbourhood $\mathcal{U}_{\alpha,\beta}$ of $R_{\alpha,\beta}$ in $D_{i(\alpha)}$ such that
\begin{equation*}
\tau_{\alpha,\beta} \circ \pi_A \circ p_{\mathbf{k}}((\alpha_m)_{m \in \mathbb{Z}}) = t_{\mathbf{k}}((\alpha_m)_{m \in \mathbb{Z}}) 
\end{equation*}
for every $(\alpha_m)_{m \in \mathbb{Z}} \in \Sigma_{\mathbf{k}}$ such that $\alpha_0 = \alpha$ and $\alpha_1 = \beta$, and
\begin{equation*}
\phi_{\tau_{\alpha,\beta}(x)}(x) \in D_{i(\beta)}
\end{equation*}
for every $x \in \mathcal{U}_{\alpha,\beta}$. By takin $\mathcal{U}_{\alpha,\beta}$ small enough, we may assume that $\tau_{\alpha,\beta}$ is valued in $(0,2\kappa)$. We let then $F_{\alpha,\beta} : \mathcal{U}_{\alpha,\beta} \to D_{i(\beta)}$ be defined by $F_{\alpha,\beta}(x) = \phi_{\tau_{\alpha,\beta}(x)}(x)$. Notice that the map $F_{\alpha,\beta}$ is $G^s$. This is precisely for this reason that we need to work with $G^s$ Markov partitions.

If $x \in D_i \cap K$ for some $i \in I$, let $E_{\un,i}(x) = T_x D_i \cap (E_\un(x) \oplus X(x))$ and $E_{\st,i}(x) = T_x D_i \cap (E_\st(x) \oplus X(x))$. Notice that if $A_{\mathbf{k}}(\alpha,\beta) = 1$ and $x \in R_{\alpha,\beta}$ then $D F_{\alpha,\beta}(x) E_{\un,i(\alpha)}(x) = E_{\un,i(\beta)}(F_{\alpha,\beta}(x))$ and $D F_{\alpha,\beta}(x) E_{\st,i(\alpha)}(x) = E_{\st,i(\beta)}(F_{\alpha,\beta}(x))$. We have $T_x D_{i(\alpha)} = E_{\un,i(\alpha)}(x) \oplus E_{\st,i(\alpha)}(x)$, but we will rather work with the dual decomposition $T_x^* D_i = E_{\un,i}^*(x) \oplus E_{\st,i}^*(x)$ where $E_{\un,i}^*(x)$ and $E_{\st,i}^*(x)$ are the annihilators of $E_{\un,i}(x)$ and $E_{\st,i}(x)$ respectively. We will use a norm on $T_x D_i$ defined in the following way: we choose a Mather metric\footnote{A Mather metric is a \emph{H\"older-continuous} Riemannian metric on $M$ that makes for every $x \in K$ the splitting \eqref{eq:hyperbolic_splitting} orthogonal and for which one may take $C = 1$ in the last point of Definition \ref{definition:hyperbolic}. Such a metric is defined on the stable and unstable directions by averaging any metric under the action of the flow for some large time, respectively in forward and backard time. See \cite{mather_norm}.} on $M$, it induces a metric on $E_\un(x) \oplus E_\st(x)$ that we identify then with $T_x D_i$ by projecting along the flow direction. The advantage of this definition is that if $x \in R_{\alpha,\beta}$ then, for the corresponding operator norms, $\n{D F_{\alpha,\beta}(x)_{| E_{\un,i(\alpha)}(x)}^{-1}}$ and $\n{DF_{\alpha,\beta}(x)_{|E_{\st,i(\alpha)}(x)}}$ are strictly less than $1$.

Writing $\mathcal{D} = \bigcup_{i \in I} D_i$, we can then introduce an escape function $G$ on $T^*_{K} \mathcal{D} = \bigcup_{x \in \mathcal{D} \cap K } T_x^* \mathcal{D}$, defined for $(x, \xi) \in T^* \mathcal{D}$ by
\begin{equation*}
G(x,\xi) = |\xi_\st|^{\frac{1}{s}} - |\xi_\un|^{\frac{1}{s}}.
\end{equation*}
Here, $\xi = \xi_{\un} + \xi_{\st}$ is the decomposition of $\xi$ according to the sum $T_x^* D_i = E_{\un,i}^*(x) \oplus E_{\st,i}^*(x)$ where $x \in D_i$, and we used the norms on $T_x^* D_i$ induced by the norm we just defined on $T_x D_i$. Beware that $G$ is a priori not smooth (it is continuous however) since $E_{\un,i}^*(x)$ and $E_{\st,i}^*(x)$ do not depend smoothly on $x$ in general. Notice that there is $C > 0$ such that for every $x \in R_{\alpha,\beta}$ and $\xi \in T^*_{x} D_{i(\alpha)}$, we have\footnote{We write $A^{-\top}$ for the transpose of the inverse of $A$.}
\begin{equation}\label{eq:decay_escape_function}
G(F_{\alpha,\beta}(x), DF_{\alpha,\beta}(x)^{-\top} \xi) - G(x, \xi)  \leq - C^{-1}|\xi|^{\frac{1}{s}}.
\end{equation}
This formula suggests to introduce the symplectic lift $\mathcal{F}_{\alpha,\beta}$ of $F_{\alpha,\beta}$, defined for $(x,\xi) \in T^* \mathcal{U}_{\alpha,\beta}$ by
\begin{equation*}
\mathcal{F}_{\alpha,\beta} (x,\xi) = (F_{\alpha,\beta}(x),(D F_{\alpha,\beta}(x))^{-\top}\xi).
\end{equation*}

Let us identify the $D_i$'s with disjoint balls in $\mathbb{R}^{n-1}$ and introduce the Kohn--Nirenberg distance on $T^* \mathbb{R}^{n-1}$. We introduce the metric $g_{KN}$ on $T^* \mathbb{R}^{n-1} \simeq \mathbb{R}^{n-1}_x \times \mathbb{R}^{n-1}_\xi$ by
\begin{equation*}
g_{KN} = \mathrm{d}x^2 + \frac{\mathrm{d}\xi^2}{1 + |\xi|^2}
\end{equation*}
and let $d_{KN}$ be the distance associated to this metric. Notice that two points $(x,\xi)$ and $(y,\eta)$ of $T^* \mathbb{R}^{n-1}$ are close for $d_{KN}$ is the Euclidean distance between $x$ and $y$ is small, $\brac{\xi}$ and $\brac{\eta}$ have the same order of magnitude\footnote{We use here the Japanese bracket notation $\brac{\xi} = \sqrt{1 + |\xi|^2}$.}, and the Euclidean distance between $\xi$ and $\eta$ is small in front of this order of magnitude.

It follows from \eqref{eq:decay_escape_function} and the continuity of the stable and unstable direction of $\phi$ that there are $\varpi > 0$ and $C > 0$ such that for every $x \in R_\alpha , y \in \mathcal{U}_{\alpha,\beta}, z \in R_\beta$ and $\xi,\eta, \mu \in \mathbb{R}^{n-1}$, if $d_{KN}((x,\xi),(y,\eta)) \leq \varpi$, $d_{KN}(\mathcal{F}_{\alpha,\beta}(y,\eta), (z,\mu)) \leq \varpi$ and $|\mu| \geq 1$ then
\begin{equation}\label{eq:approximate_decay}
G(z,\mu) - G(x,\xi) \leq - C^{-1} |\mu|^{\frac{1}{s}}.
\end{equation}

Let us now state a lemma that we need to fix some parameter in prevision for the proof of Lemma \ref{lemma:value_trace}.

\begin{lemma}\label{lemma:useful_later}
There is $\delta > 0$ such that the following holds. Let $m \geq 1$ be an integer. Let $\alpha_0,\dots,\alpha_{m-1} \in I_{\mathbf{k}}$ be such that $A_{\mathbf{k}}(\alpha_0,\alpha_1) = \dots = A_{\mathbf{k}}(\alpha_{m-2},\alpha_{m-1}) = A_{\mathbf{k}}(\alpha_{m-1},\alpha_0) = 1$. For $j = 0,\dots,m-1$, let $a_j \in R_{\alpha_j}$ be such that $\overline{B(a_j,\delta)} \cap R_{\alpha_j,\alpha_{j+1}} \neq \emptyset$ and let $x_j \in \overline{B(a_j,\delta)}$. For $j = 0,\dots,m-2$, assume that we have $ F_{\alpha_j,\alpha_{j+1}}(x_j) = x_{j+1}$. Assume also that $F_{\alpha_{m-1},\alpha_0}(x_{m-1}) \in \overline{B(a_0,\delta)}$. Then $1$ is not an eigenvalue of $DF_{\alpha_{m-1},\alpha_0}(x_{m-1}) \dots D F_{\alpha_0,\alpha_1}(x_0)$.
\end{lemma}

\begin{remark}
In the statement of Lemma \ref{lemma:useful_later}, we use $B(x,\delta)$ to denote the Euclidean ball of center $x$ and radius $\delta$ in $\mathbb{R}^{n-1}$. It is crucial in this lemma that we identified the $D_i$'s with balls in $\mathbb{R}^{n-1}$. This is because the derivative $DF_{\alpha_{m-1},\alpha_0}(x_{m-1}) \dots D F_{\alpha_0,\alpha_1}(x_0)$ is not intrinsically an endomorphism but is identified with a $(n-1) \times (n-1)$ matrix using the standard identification of the tangent space of $\mathbb{R}^{n-1}$ (at any point) with $\mathbb{R}^{n-1}$.
\end{remark}

\begin{proof}[Proof of Lemma \ref{lemma:useful_later}]
Let $v_0 \in \mathbb{R}^{n-1}$ be such that $$DF_{\alpha_{m-1},\alpha_0}(x_{m-1}) \dots D F_{\alpha_0,\alpha_1}(x_0) v_0 = v_0.$$ For $j = 1,\dots, m-1$, define $v_j = DF_{\alpha_{j-1},\alpha_j}(x_{j-1}) \dots D F_{\alpha_0,\alpha_1}(x_0) v_0$. For $j = 0,\dots,m-1$, notice that since we identified $D_j$ with a disk in $\mathbb{R}^{n-1}$, then $E_{\st,j}(a_j)$ and $E_{\un,j}(a_j)$ identify with subspaces of $\mathbb{R}^{n-1}$. Let us consequently split $v_j = v_{\st,j} + v_{\un,j}$ according to the decomposition $\mathbb{R}^{n-1} = E_{\st,j}(a_j) \oplus E_{\un,j}(a_j)$.

We have then for $j = 0,\dots, m-1$ (using a periodic indexing in the case $j = m-1$) that
\begin{equation*}
\begin{split}
v_{j+1} & = DF_{\alpha_{j}, \alpha_{j+1}}(x_j)v_j \\ & = (DF_{\alpha_j,\alpha_{j+1}}(x_j) - DF_{\alpha_j,\alpha_{j+1}}(a_j))v_j + DF_{\alpha_j,\alpha_{j+1}}(a_j) v_{\st,j} \\ & \qquad \qquad \qquad \qquad \qquad \qquad \qquad \qquad \qquad \qquad + DF_{\alpha_j,\alpha_{j+1}}(a_{j})v_{\un,j}.
\end{split}
\end{equation*}
In the following $C > 0$ denotes a constant that does not depend on $\delta$, nor on the other data from the statement of the lemma, and may change from one line to another. Using that $DF_{\alpha_j,\alpha_{j+1}}$ is smooth, we get
\begin{equation*}
|(DF_{\alpha_j,\alpha_{j+1}}(x_j) - DF_{\alpha_j,\alpha_{j+1}}(a_j))v_j| \leq C \delta |v_j|.
\end{equation*}
Notice that $DF_{\alpha_j,\alpha_{j+1}}(a_j) v_{\st,j}$ and $DF_{\alpha_j,\alpha_{j+1}}(a_{j})v_{\un,j}$ belong respectively to $E_{\st,j+1}(F_{\alpha_j,\alpha_{j+1}}(a_j))$ and $E_{\un,j+1}(F_{\alpha_j,\alpha_{j+1}}(a_j))$. Since $a_{j+1}$ is at distance at most $C\delta$ from $F_{\alpha_j,\alpha_{j+1}}(a_j)$, the stable and unstable directions are H\"older-continuous and $v_{\un,j+1}$ is the projection of $v_{j+1}$ on $E_{\un,j+1}(a_{j+1})$ along the direction $E_{\st,j+1}(a_j)$, we get, for some $\varrho > 0$ and $\lambda > 1$, that
\begin{equation}\label{eq:almost_unstable_property}
\begin{split}
|v_{\un,j+1}| & \geq (1 - C \delta^{\varrho})|DF_{\alpha_j,\alpha_{j+1}}(a_{j})v_{\un,j}| - C \delta^\varrho |DF_{\alpha_j,\alpha_{j+1}}(a_j) v_{\st,j}| \\ & \qquad \qquad \qquad \qquad \qquad \qquad - |(DF_{\alpha_j,\alpha_{j+1}}(x_j) - DF_{\alpha_j,\alpha_{j+1}}(a_j))v_j| \\
          & \geq \lambda( 1 - C \delta^{\varrho}) |v_{\un,j}| - C \delta^\varrho |v_j|.
\end{split}
\end{equation}
Here, we are using the Mather metric at $a_j$ to measure all quantities with a $j$ index. This is how we get $\lambda > 1$, we must however take into account the dependence of the metric on the point (since we do not have $F_{\alpha_j,\alpha_{j+1}}(a_j) = a_{j+1}$), this is included in the correction factor $1 - C \delta^\varrho$: the Mather metric is H\"older-continous. A similar reasoning yields
\begin{equation}\label{eq:almost_stable_property}
|v_{\st,j+1}| \leq \lambda^{-1}(1 + C \delta^\varrho) |v_{\st,j}| + C \delta^\varrho|v_j|.
\end{equation}

Let us then choose $\tau \in(0,\min(1 - \lambda^{-2},1/2))$ and assume that $\delta$ is small enough to ensure that
\begin{equation*}
1 + \frac{\lambda^{-1}(1 + C \delta^{\varrho}) + C \delta^{\varrho}}{\tau \lambda(1 - C \delta^{\varrho}) - C \delta \varrho} < \tau^{-1} \textup{ and } \lambda(1 - C \delta^\varrho) - C \tau^{-1} \delta^\varrho > 1,
\end{equation*}
with the constant $C$ from \eqref{eq:almost_unstable_property} and \eqref{eq:almost_stable_property}. Now, notice that if there is $j \in \set{0,\dots,m-1}$ such that $|v_{\un,j}| > \tau |v_j|$ then it follows from \eqref{eq:almost_unstable_property} and \eqref{eq:almost_stable_property} that $|v_{\un,j+1}| > \tau |v_{j+1}|$ and $|v_{\un,j+1}| \geq (\lambda (1-C \delta^\varrho) - C \tau^{-1} \delta^\varrho)|v_{\un,j}|$. Iterating this reasoning $m$ times and recalling that we use periodic indexing, we find that
\begin{equation*}
|v_{\un,j}| \geq (1-C \delta^\varrho - C \tau^{-1} \delta^\varrho)^m|v_{\un,j}| > |v_{\un,j}|,
\end{equation*}
a contradiction. Hence, we have $|v_{\un,j}| \leq \tau |v_j|$ for $j = 0,\dots,m-1$. Similarly, we also get $|v_{\st,j}| \leq \tau|v_j|$. Since $\tau < \frac{1}{2}$, it follows that $v_0 = 0$.
\end{proof}

Let us now fix $\delta > 0$ small enough so that the following holds:
\begin{itemize}
\item $\overline{B(0,\delta)} \subseteq ] - 1, 1[^{n-1}$;
\item for every $\ell \in \mathbb{Z}^{n-1}$ the diameter of 
\begin{equation*}
\set{(y,2 \pi \ell) : y \in B(0,\delta)}
\end{equation*}
for the Kohn--Nirenberg distance is less than $\varpi/10$ (where $\varpi$ is defined above \eqref{eq:approximate_decay});
\item for every $\alpha \in I_{\mathbf{k}}$ and $x \in D_{i(\alpha)}$, if there is $\beta \in I_{\mathbf{k}}$ such that $A_{\mathbf{k}}(\alpha,\beta) = 1$ and $\overline{B(x,\delta)}$ intersects $R_{\alpha,\beta}$, then $\overline{B(x,\delta)}$ is contained in $\mathcal{U}_{\alpha,\beta}$;
\item Lemma \ref{lemma:useful_later} holds.
\end{itemize}
For each $\alpha \in I_{\mathbf{k}}$, let us fix a finite set $P_{\alpha}$ of points in $R_{\alpha}$ such that 
\begin{equation*}
\mathcal{V}_\alpha \coloneqq \bigcup_{x \in P_{\alpha}} B(x,\delta/2)
\end{equation*}
covers $R_{\alpha}$. Choose then two families of $G^s$ functions $(\theta_a)_{a \in P_{\alpha}}, (\tilde{\theta}_a)_{a \in P_{a}}$ such that $\sum_{a \in P_{\alpha}} \theta_a \equiv 1$ on $\mathcal{V}_\alpha$ and for each $a \in P_\alpha$, the functions $\theta_a$ and $\tilde{\theta}_a$ are supported in $B(a,2\delta/3)$ and $\tilde{\theta}_a \equiv 1$ on a neighbourhood of the support of $\theta_a$.

Let then $e_1,\dots,e_d$ denote the canonical basis of $\mathbb{C}^d$ and define for $a \in P_{\alpha}, \ell \in \mathbb{Z}^{n-1}$ and $j \in \set{1,\dots,d}$ the functions
\begin{equation*}
\mathbf{e}_{a,\ell,j} : x \mapsto \theta_a(x) e^{2 i \pi \ell \cdot x} e_j \textup{ and } \tilde{\mathbf{e}}_{a,\ell,j} : x \mapsto \tilde{\theta}_a(x) e^{2 i \pi \ell \cdot x} e_j.
\end{equation*}
Notice then that for every smooth function $f : \mathbb{R}^{n-1} \to \mathbb{C}^d$ supported in $\mathcal{V}_\alpha$, the Poisson formula implies that
\begin{equation}\label{eq:inverse_Fourier}
f = \sum_{a \in P_{\alpha}} \sum_{\ell \in \mathbb{Z}^{n-1}} \sum_{j = 1}^d \langle f, \mathbf{e}_{a,\ell,j} \rangle \tilde{\mathbf{e}}_{a,\ell,j},
\end{equation}
where $\langle \cdot, \cdot \rangle$ denote the standard $L^2$ scalar product on $\mathbb{R}^{n-1}$. For $\epsilon > 0$ and $f \in C^\infty(\mathbb{R}^{n-1})$, define then
\begin{equation*}
\n{f}_{\epsilon,\alpha}^2 = \sum_{a \in P_{\alpha}} \sum_{\ell \in \mathbb{Z}^{n-1}} \sum_{j = 1}^d e^{- 2 \epsilon G(a,\ell)}|\langle f, \mathbf{e}_{a,\ell,j} \rangle|^2 \in \mathbb{R}_+ \cup \set{+ \infty}.
\end{equation*}
Let then $\mathcal{H}_{\epsilon,\alpha}$ be the completion of $$\set{f \in C_c^\infty(\mathbb{R}^{n-1}) : \supp f \subseteq \mathcal{V}_\alpha \textup{ and } \n{f}_{\epsilon,\alpha} < + \infty}$$ for the norm $\n{\cdot}_{\epsilon,\alpha}$. Having defined this space for each $\alpha \in I_{\mathbf{k}}$, we set then 
\begin{equation*}
\mathcal{H}_\epsilon = \bigoplus_{\alpha \in I_{\mathbf{k}}} \mathcal{H}_{\epsilon,\alpha}.
\end{equation*}

For every $i \in I$, the disk $D_i$ is simply connected, and thus the pullback of the vector bundle $E \to M$ by the injection $D_i \hookrightarrow M$ is trivial. Hence, for each $x \in D_i$ there is a linear isomorphism $H_{i,x}$ between $\mathbb{C}^d$ and the fiber $E_{x}$ such that the map $(x,v) \to H_{i,x}(v)$ from $D_i \times \mathbb{C}^d$ to $E$ is $G^s$. If $\alpha,\beta \in I_{\mathbf{k}}$ are such that $A_{\mathbf{k}}(\alpha,\beta) = 1$, we define for every $x \in \mathcal{U}_{\alpha,\beta}$ the matrix
\begin{equation}\label{eq:action_fiber}
N_{\alpha,\beta}(x) =\left( H_{i(\beta),F_{\alpha,\beta}(x)}^{-1} \Phi_{\tau_{\alpha,\beta}(x)} H_{i(\alpha),x}\right)^\top.
\end{equation}
Let then choose a $G^s$ function $\chi_{\alpha,\beta}$ compactly supported in $\mathcal{U}_{\alpha,\beta} \cap \mathcal{V}_\alpha$ and such that $\chi_{\alpha,\beta} \equiv 1$ on a neighbourhood of $R_{\alpha,\beta}$, and introduce for $z \in \mathbb{C}$ the operator $L_{z,\alpha,\beta}$ defined on $C^\infty(D_{i(\beta)},\mathbb{C}^d)$ by
\begin{equation*}
L_{z,\alpha,\beta} u (x) = \chi_{\alpha,\beta}(x) e^{-z \tau_{\alpha,\beta}(x)} N_{\alpha,\beta}(x) u (F_{\alpha,\beta}(x)).
\end{equation*}

The $L_{z,\alpha,\beta}$'s are the basic blocks for the construction of an operator $\mathcal{L}_z$ whose Fredholm determinant is $d(z)$ (see \eqref{eq:definition_entry} and \eqref{eq:abstract_operator}). However, before defining $\mathcal{L}_z$, we need to establish basic properties of the $L_{z,\alpha,\beta}$'s. This is the point of the following two lemmas.

\begin{lemma}\label{lemma:non_stationary_phase}
There is a constant $C > 0$ such that for every $\alpha,\beta \in I_{\mathbf{k}}$ such that $A_{\mathbf{k}}(\alpha,\beta) = 1$, every $a \in P_{\beta}$, every $b \in P_{\alpha}$, every $\ell,\ell' \in \mathbb{Z}^{n-1}$, every $p,p' \in \set{1,\dots,d}$ and every $z \in \mathbb{C}$, we have
\begin{equation}\label{eq:general_upper_bound_coeff}
|\langle  L_{z,\alpha,\beta} \tilde{\mathbf{e}}_{a,\ell,p}, \mathbf{e}_{b,\ell',p'} \rangle| \leq C e^{C|z|}.
\end{equation} 
If in addition the sets $\set{(y,2 \pi \ell) : y \in B(a,\delta)}$ and $\mathcal{F}_{\alpha,\beta}(\set{(y,2 \pi \ell') : y \in B(b,\delta)})$ are at distance larger than $\varpi/10$ then
\begin{equation}\label{eq:better_upper_bound_coeff}
|\langle  L_{z,\alpha,\beta} \tilde{\mathbf{e}}_{a,\ell,p}, \mathbf{e}_{b,\ell',p'} \rangle| \leq C e^{C|z| - C^{-1} \max(|\ell|,|\ell'|)^{\frac{1}{s}}}
\end{equation}
\end{lemma}

\begin{proof}
Let us write
\begin{equation}\label{eq:explicit_scalar_product}
\langle  L_{z,\alpha,\beta} \tilde{\mathbf{e}}_{a,\ell,p}, \mathbf{e}_{b,\ell',p'} \rangle = \int_{\mathbb{R}^{n-1}} h_z(x) e^{2 i \pi \Phi_{\alpha,\beta}^{\ell,\ell'}(x)} \mathrm{d}x,
\end{equation}
where 
\begin{equation*}
\Phi_{\alpha,\beta}^{\ell,\ell'}(x) = \ell \cdot F_{\alpha,\beta}(x) - \ell' \cdot x
\end{equation*}
and $h_z(x) = e^{-z \tau_{\alpha,\beta}(x)} g(x)$ with
\begin{equation*}
g(x) = \chi_{\alpha,\beta}(x) \theta_b(x) \tilde{\theta}_a(F_{\alpha,\beta}(x)) \langle N_{\alpha,\beta}(x) e_p, e_{p'} \rangle.
\end{equation*}
Notice that $g$ is a $G^s$, the $G^s$ estimate being uniform in all the parameters of the problem.

The estimate \eqref{eq:general_upper_bound_coeff} follows directly by bounding $h_z$ by its supremum norm (and using that the size of its support is bounded). To prove the estimate \eqref{eq:better_upper_bound_coeff}, we assume that the Kohn--Nirenberg distance between $\set{(y,2 \pi \ell) : y \in B(a,\delta)}$ and $\mathcal{F}_{\alpha,\beta}(\set{(y,2 \pi \ell') : y \in B(b,\delta)})$ is greater than $\varpi/10$ and we compute for $x$ in the support of $h_z$:
\begin{equation}\label{eq:non_stationary}
\begin{split}
|\nabla \Phi_{\alpha,\beta}^{\ell,\ell'}(x)| & = |D F_{\alpha,\beta}^\top (x) \ell - \ell'| \\ 
   & \geq C^{-1} |\ell - D F_{\alpha,\beta}(x)^{-\top} \ell'| \\
   & \geq C^{-1} \max(|\ell|,|\ell'|).
\end{split}
\end{equation}
Here, we used the fact that the distance between the points $(F_{\alpha,\beta}(x), \ell)$ and $(F_{\alpha,\beta}(x), DF_{\alpha,\beta}(x)^{-\top} \ell')$ for the Kohn--Nirenberg metric is bounded from below (by $\varpi/10$), whenever $x$ is in the support of $h_z$ (because it imposes that $F_{\alpha,\beta}(x) \in B(a,\delta)$). The estimate \eqref{eq:non_stationary} expresses the fact that the phase in \eqref{eq:explicit_scalar_product} is non-stationary. We will consequently end the proof by a Gevrey non-stationary phase argument.

Indeed, we can rewrite \eqref{eq:explicit_scalar_product} as (ignoring the case in which $\max(|\ell|,|\ell'|) = 0$)
\begin{equation*}
\int_{\mathbb{R}^{n-1}} h_z(x) e^{\max(|\ell|,|\ell'|) \frac{2 i \pi \Phi_{\alpha,\beta}^{\ell,\ell'}(x)}{\max(|\ell|,|\ell'|)}} \mathrm{d}x
\end{equation*}
and then apply \cite[Proposition 1.37]{asterisque} with large parameter $\max(|\ell|,|\ell'|)$. Notice that the condition that the imaginary part of the phase is positive on the boundary of the domain of integration from \cite[Proposition 1.37]{asterisque} is not satisfied here, but since the function $h_z$ is compactly supported, we could just make it happen artificially, so that the result still applies. A bound of the form $\mathcal{O}(e^{C|z|})$ on a Gevrey norm of $h_z$ may be obtained for instance by \cite[Lemma 1.12]{asterisque}.
\end{proof}

We continue our study of the operators $L_{z,\alpha,\beta}$'s. Notice that the following lemma also implies that, if $\epsilon > 0$ is small enough, then the space $\mathcal{H}_\epsilon$ is non-empty.

\begin{lemma}\label{lemma:matrix_entry}
Let $\epsilon > 0$ be small enough. There is a constant $C > 0$ such that for every $\alpha,\beta \in I_{\mathbf{k}}$ such that $A_{\mathbf{k}}(\alpha,\beta) = 1$, every $a \in P_{\beta}$, every $\ell \in \mathbb{Z}^{n-1}$, every $p \in \set{1,\dots,d}$ and every $z \in \mathbb{C}$, we have
\begin{equation}\label{eq:size_matrix_entry}
\n{L_{z,\alpha,\beta} \tilde{\mathbf{e}}_{a,\ell,p}}_{\epsilon,\alpha} \leq C e^{ - \epsilon G(a,\ell) - C^{-1}|\ell|^{\frac{1}{s}} + C |z|}.
\end{equation}
\end{lemma}

\begin{proof}
The definition of the norm is
\begin{equation*}
\n{L_{z,\alpha,\beta} \tilde{\mathbf{e}}_{a,\ell,p}}_{\epsilon,\alpha}^2 = \sum_{b \in P_{\alpha}} \sum_{\ell \in \mathbb{Z}^{n-1}} \sum_{p' = 1}^d e^{- 2 \epsilon G(b,\ell')}|\langle L_{z,\alpha,\beta} \tilde{\mathbf{e}}_{a,\ell,p}, \mathbf{e}_{b,\ell,p'} \rangle|^2.
\end{equation*}
We will denote by $\mathcal{E}_1$ the set of $(b,\ell',p') \in P_{\alpha} \times \mathbb{Z}^{n-1} \times \set{1,\dots,d}$ such that the distance between $\set{(y,2 \pi \ell) : y \in B(a,\delta)}$ and $\mathcal{F}_{\alpha,\beta}(\set{(y,2 \pi \ell') : y \in B(b,\delta)})$ is less than or equal to $\varpi/10$ and $\mathcal{E}_2$ the complement of $\mathcal{E}_1$ in $P_{\alpha} \times \mathbb{Z}^{n-1} \times \set{1,\dots,d}$.

It follows from Lemma \ref{lemma:non_stationary_phase} that
\begin{equation*}
\begin{split}
& \sum_{(b,\ell',p') \in \mathcal{E}_2} e^{- 2 \epsilon G(b,\ell')}|\langle L_{z,\alpha,\beta} \tilde{\mathbf{e}}_{a,\ell,p}, \mathbf{e}_{b,\ell,p'} \rangle|^2 \\ & \qquad \qquad \qquad \leq C e^{C|z|} \sum_{(b,\ell',p') \in P_{\alpha} \times \mathbb{Z}^{n-1} \times \set{1,\dots,d}} e^{-2 \epsilon G(b,\ell') - C^{-1} \max(|\ell|,|\ell'|)^{\frac{1}{s}}} \\
 & \qquad \qquad \qquad \leq C e^{C|z| - 2 \epsilon G(a,\ell)} \sum_{(b,\ell',p') \in P_{\alpha} \times \mathbb{Z}^{n-1} \times \set{1,\dots,d}} e^{- C^{-1} \max(|\ell|,|\ell'|)^{\frac{1}{s}}}  \\
 & \qquad \qquad \qquad \leq C e^{C|z|- 2 \epsilon G(a,\ell) - C^{-1} |\ell|^{\frac{1}{s}}},
\end{split}
\end{equation*}
where the constant $C$ may change from one line to another. In particular to go from the second to the third line, where we assume that $\epsilon$ is small enough.

Assume now that $(b,\ell',p') \in \mathcal{E}_1$. Let $y \in B(b,\delta)$ be such that the distance between $\mathcal{F}_{\alpha,\beta}(y,2 \pi \ell')$ and $(a, 2 \pi \ell)$ is less than $\varpi/5$ (recall that the diameter of $\set{(y,2 \pi \ell) : y \in B(a,\delta)}$ is less than $\varpi/10$). If $\ell \neq 0$, we can apply \eqref{eq:approximate_decay} (with $x = b, z = a, \xi = \eta = 2 \pi \ell'$ and $\mu = 2 \pi \ell$), to find that, for another value of $C > 0$,
\begin{equation*}
G(a, \ell) - G(b,\ell') \leq - C^{-1} |\ell|^{\frac{1}{s}}.
\end{equation*}
It follows then from \eqref{eq:general_upper_bound_coeff} that 
\begin{equation*}
\begin{split}
& e^{- 2 \epsilon G(b,\ell')}|\langle L_{z,\alpha,\beta} \tilde{\mathbf{e}}_{a,\ell,p}, \mathbf{e}_{b,\ell,p'} \rangle|^2 \\ & \qquad \qquad \qquad \qquad \leq C e^{C|z| - 2 \epsilon G(a,\ell)} e^{2 \epsilon (G(a,\ell) - G(b,\ell')} \\ & \qquad \qquad \qquad \qquad \leq C e^{C |z| - 2 \epsilon G(a,\ell) - C^{-1} |\ell|^{\frac{1}{s}}}. 
\end{split} 
\end{equation*} 
Since the cardinal of $\mathcal{E}_1$ grows at most polynomially with $|\ell|$, it establishes \eqref{eq:size_matrix_entry} when $\ell$ is non-zero (summing over $\mathcal{E}_1$ only has the effect of making $C$ larger). The proof in the case $\ell = 0$ is actually simpler since both $|\ell|^{1/s} = G(a,\ell) = 0$ and $\ell'$ stays within a bounded set.
\end{proof}

Now, if $\alpha$ and $\beta$ are such that $A_{\mathbf{k}}(\alpha,\beta) = 1$, we define an operator $\mathcal{L}_{z,\alpha,\beta}$ going from $\mathcal{H}_{\epsilon, \beta}$ to $\mathcal{H}_{\epsilon,\alpha}$ by
\begin{equation}\label{eq:definition_entry}
\mathcal{L}_{z,\alpha,\beta} f = \sum_{a \in P_{\beta}} \sum_{\ell \in \mathbb{Z}^{n-1}} \sum_{p = 1}^d \langle f, \mathbf{e}_{a,\ell,p} \rangle L_{z,\alpha,\beta} \tilde{\mathbf{e}}_{a,\ell,p}.
\end{equation}
It follows from Lemma \ref{lemma:matrix_entry} that this operator is well-defined if $\epsilon > 0$ is sufficiently small. We fix such an $\epsilon$ from now on. Let then $\mathcal{L}_z$ be the operator acting on $\mathcal{H}_\epsilon$ defined for $f = (f_{\alpha})_{\alpha \in I_{\mathbf{k}}}$ by
\begin{equation}\label{eq:abstract_operator}
(\mathcal{L}_z f)_{\alpha} = \sum_{\substack{\beta \in I_{\mathbf{k}} \\ A_{\mathbf{k}}(\alpha,\beta) = 1}} \mathcal{L}_{z,\alpha,\beta} f_{\beta}.
\end{equation}

We will see later (Lemma \ref{lemma:identification_determinant}) that the determinant $\det(I - \mathcal{L}_z)$ is the holomorphic continuation of $d_{\mathbf{k}}(z)$. Lemma \ref{lemma:individual_estimates} will then be a consequence of the following bound:

\begin{lemma}\label{lemma:abstract_determinant}
For every $z \in \mathbb{C}$, the operator $\mathcal{L}_z$ is trace class. Moreover, the map $z \mapsto \det(I - \mathcal{L}_z)$ is holomorphic and there is a constant $C > 0$ such that for every $z \in \mathbb{C}$ we have
\begin{equation}\label{eq:bound_abstract_determinant}
|\det(I - \mathcal{L}_z)| \leq C \exp( C |z|^{1 + (n-1)s}).
\end{equation}
\end{lemma}

\begin{proof}
Let us start by noticing that $\mathcal{L}_z$ may be written as\footnote{We write $f_m \otimes l_m$ for the rank $1$ operator $x \mapsto l_m(x) f_m$.}
\begin{equation}\label{eq:sum_rank_one}
\mathcal{L}_z = \sum_{m \geq 0} \lambda_m f_m \otimes l_m
\end{equation}
where $f_m \in \mathcal{H}_{\epsilon}$ and $l_m \in \mathcal{H}_\epsilon^*$ have norm $1$, while the sequence of complex numbers $(\lambda_m)_{m \geq 0}$ satisfies
\begin{equation}\label{eq:bound_coefficients}
|\lambda_m| \leq C e^{C|z| - C^{-1} m^{\frac{1}{s(n-1)}}}
\end{equation} 
for some $C > 0$ (that does not depend on $z$) and every $m \geq 0$. Indeed, each of the entry $\mathcal{L}_{z,\alpha,\beta}$ has this property (according to their definition \eqref{eq:definition_entry} and Lemma \ref{lemma:matrix_entry}), so that $\mathcal{L}_z$ can be written in this form just by relabelling. The exponent $1/s$ from Lemma \ref{lemma:matrix_entry} has been replaced here by $1/(s(n-1))$ because we included the summation over $\ell \in \mathbb{Z}^{n-1}$ into the single index $m$.

Since $\mathcal{L}_z$ is the sum of an absolutely convergent series of rank $1$ operator, it is trace class. To prove that $z \mapsto \det(I - \mathcal{L}_z)$ is holomorphic, we only need to prove that $z \mapsto \mathcal{L}_z$ is holomorphic (valued in the space of trace class operator) and thus that each term in \eqref{eq:sum_rank_one} is holomorphic in $z$ (since the series converges uniformly locally in $z$). Comparing with \eqref{eq:definition_entry}, we see that we want to prove that for every $\alpha,\beta \in I_{\mathbf{k}}$ such that $A_{\mathbf{k}}(\alpha,\beta) = 1$, every $a \in P_{\beta}$, every $\ell \in \mathbb{Z}^{n-1}$ and every $p \in \set{1,\dots,d}$ the map $z \mapsto L_{z,\alpha,\beta} \tilde{\mathbf{e}}_{a,\ell,p}$ (from $\mathbb{C}$ to $\mathcal{H}_{\epsilon,\alpha}$) is holomorphic. 

Let us start by noticing that for every $b \in P_{\alpha}, \ell' \in \mathbb{Z}^{n-1}$ and $p' \in \set{1,\dots,d}$ the function $z \mapsto \langle L_{z,\alpha,\beta} \mathbf{e}_{a,\ell,p} , \mathbf{e}_{b,\ell',p'} \rangle$ is continuous (it is even holomorphic by differentiation under the integral). For $z \in \mathbb{C}$, we can bound the terms appearing in the norm $\n{( L_{z,\alpha,\beta} - L_{z+h,\alpha,\beta}) \mathbf{e}_{a,\ell,p}}^2$, with $h$ small, as in the proof of Lemma \ref{lemma:matrix_entry}, and apply the dominated convergence theorem to prove that the map $z \mapsto L_{z,\alpha,\beta} \mathbf{e}_{a,\ell,p}$ is continuous as a map from $\mathbb{C}$ to $\mathcal{H}_{\epsilon,\alpha}$. Pick $z_0 \in \mathbb{C}$ and some $\delta > 0$ and define the integral
\begin{equation*}
g(z) = \frac{1}{2 i \pi} \int_{\partial \mathbb{D}(z_0,\delta)} \frac{L_{w,\alpha,\beta} \mathbf{e}_{a,\ell,p}}{w-z}\mathrm{d}w
\end{equation*}
for $z \in \mathbb{D}(z_0,\delta)$. We proved the continuity of the integrand in order to justify the definition of the integral. Notice that the function $g$ is holomorphic (by differentiation under the integral). Moreover, the map $z \mapsto L_{z,\alpha,\beta} \mathbf{e}_{a,\ell,p}$ is also continuous from $\mathbb{C}$ to $C^\infty(\mathbb{R}^{n-1})$ and it follows that the function $g$ actually takes value in the subset $\set{f \in C_c^\infty(\mathbb{R}^{n-1}) : \supp f \subseteq V_\alpha \textup{ and } \n{f}_{\epsilon,\alpha} < + \infty}$ of $\mathcal{H}_{\epsilon,\alpha}$. Consequently, for $z \in \mathbb{D}(z_0,\delta)$, we can evaluate $g(z)$ at a point $x \in \mathbb{R}^{n-1}$, and it follows from Cauchy's formula that $g(z)(x) = L_{z,\alpha,\beta} \mathbf{e}_{a,\ell,p}(x)$. Hence, $g$ is just the restriction of $z \mapsto L_{z,\alpha,\beta} \mathbf{e}_{a,\ell,p}$ to a $\mathbb{D}(z_0,\delta)$, which proves that $z \mapsto L_{z,\alpha,\beta} \mathbf{e}_{a,\ell,p}$ is holomorphic in a neighbourhood of $z_0$.

It remains to prove \eqref{eq:bound_abstract_determinant}. To do so, recall the formula \eqref{eq:sum_rank_one} and write
\begin{equation*}
\det(I - \mathcal{L}_z) = 1 + \sum_{p \geq 1} (-1)^p \sum_{0 \leq m_1 < \dots < m_p} \lambda_{m_1} \dots \lambda_{m_p} \det((l_{m_i}(f_{m_j})_{1 \leq i,j \leq p})
\end{equation*}
as in \cite[Chapter 2, p.13]{grothendieck_55}. Using Hadamard's inequality and \eqref{eq:bound_coefficients}, we find that
\begin{equation*}
|\det(I - \mathcal{L}_z)| \leq 1 + \sum_{p \geq 1} C^p e^{p C |z|} p^{\frac{p}{2}} \sum_{0 \leq m_1 < \dots < m_p} \exp( - C^{-1} \sum_{k = 1}^p m_k^{\frac{1}{s(n-1)}}).
\end{equation*}
Notice that 
\begin{equation*}
e^{- (2C)^{-1}(k-1)^{\frac{1}{s(n-1)}}} - e^{- (2C)^{-1} k^{\frac{1}{s(n-1)}}} \underset{k \to + \infty}{\sim} \frac{k^{\frac{1}{s(n-1)} - 1} e^{- (2C)^{-1} k^{\frac{1}{s(n-1)}}}}{2C s (n-1)},
\end{equation*}
and thus, for $k$ large enough, we have
\begin{equation*}
e^{- C^{-1} k^{\frac{1}{s(n-1)}}} \leq e^{- (2C)^{-1}(k-1)^{\frac{1}{s(n-1)}}} - e^{- (2C)^{-1} k^{\frac{1}{s(n-1)}}}.
\end{equation*}
It follows that there is a constant $A \geq 2C $ such that for every $m \geq 0$, we have
\begin{equation*}
\sum_{k \geq m} e^{ - C^{-1} k^{\frac{1}{s(n-1)}}} \leq A e^{-A^{-1} m^{\frac{1}{s(n-1)}}}.
\end{equation*}
Hence, for $p \geq 1$, we have
\begin{equation*}
\begin{split}
& \sum_{0 \leq m_1 < \dots < m_{p+1}} \exp( - C \sum_{k = 1}^{p+1} m_k^{\frac{1}{s(n-1)}}) \\ & \qquad \qquad \qquad \leq \sum_{0 \leq m_1 < \dots < m_p} \exp( - C \sum_{k = 1}^p m_k^{\frac{1}{s(n-1)}})\sum_{m > m_{p}} \exp(-C m^{\frac{1}{s(n-1)}}) \\
   & \qquad \qquad \qquad \leq A \sum_{0 \leq m_1 < \dots < m_p} \exp( - C \sum_{k = 1}^p m_k^{\frac{1}{s(n-1)}}) \exp(-A^{-1} (m_p + 1)^{\frac{1}{s(n-1)}}) \\
   & \qquad \qquad \qquad \leq A e^{-A^{-1} p^{\frac{1}{s(n-1)}}} \sum_{0 \leq m_1 < \dots < m_p} \exp( - C \sum_{k = 1}^p m_k^{\frac{1}{s(n-1)}}).
\end{split}
\end{equation*}
Hence, we find by induction that
\begin{equation*}
\begin{split}
\sum_{0 \leq m_1 < \dots < m_{p}} \exp( - C \sum_{k = 1}^{p} m_k^{\frac{1}{s(n-1)}}) & \leq A_0 A^p\exp(- A^{-1} \sum_{k= 1}^{p-1} k^{\frac{1}{s(n-1)}})\\
     & \leq B \exp( - B^{-1} p^{1 + \frac{1}{s(n-1)}}),
\end{split}
\end{equation*}
for some new constant $A_0,B > 0$ (the constant $A_0$ is just for the initialization in the induction). Hence, up to making $C$ larger, we get
\begin{equation*}
|\det(I - \mathcal{L}_z)| \leq C \sum_{p \geq 0} e^{p C|z| - C^{-1} p^{1 + \frac{1}{s(n-1)}}}.
\end{equation*}
Assume now that $|z| \geq 1$. Then, the terms with $p \geq T |z|^{n (s-1)}$ (for some large constant $T$) may be bounded by the terms of a geometric series that does not depend on $z$, and thus they only contribute by a bounded term. Each term with $p < T |z|^{s(n-1)}$ has at most the size of the right hand side of \eqref{eq:bound_abstract_determinant}. The number of these terms is polynomial in $|z|$ and can consequently be ignored (by making $C$ larger).
\end{proof}

Our goal is now to compute $\det(I - \mathcal{L}_z)$ for $\re z \gg 1$. This will be done by computing the traces of the iterates of $\mathcal{L}_z$. To this end, for each integer $m \geq 1$, we denote by $\mathfrak{P}_{\mathbf{k},m}$ the set of orbits of the suspension flow $\tilde{\psi}_{\mathbf{k}}$ on $\widetilde{\Sigma}_{\mathbf{k}}$ that corresponds to an orbit of length $m$ of the shift on $\Sigma_{\mathbf{k}}$ (see Remark \ref{remark:suspension_flow}).

\begin{lemma}\label{lemma:value_trace}
For every $z \in \mathbb{C}$ and $m \geq 1$, we have
\begin{equation*}
\tr(\mathcal{L}_z^m) =  m \sum_{\gamma \in \mathfrak{P}_{\mathbf{k},m}}  \frac{T_{\rho_{\mathbf{k}}(\gamma)}^{\#}}{T_{\rho_{\mathbf{k}}(\gamma)}} \frac{\tr(\Phi_{\rho_{\mathbf{k}}(\gamma)})}{|\det(I - \mathcal{P}_{\rho_{\mathbf{k}}(\gamma)})|} e^{- z T_{\rho_{\mathbf{k}}(\gamma)}}.
\end{equation*}
\end{lemma}

\begin{remark}\label{remark:number_antecedents}
Before starting the proof of Lemma \ref{lemma:value_trace}, let us describe the elements of $\mathfrak{P}_{\mathbf{k},m}$ for $m \geq 1$. By definition, $\gamma \in \mathfrak{P}_{\mathbf{k},m}$ corresponds to a periodic orbit $\set{x, \sigma x, \dots, \sigma^{m-1} x}$ for the shift $\sigma$ on $\Sigma_{\mathbf{k}}$ (with $\sigma^m x = x$). Notice that a point $x \in \Sigma_{\mathbf{k}}$ satisfies $\sigma^m x = x$ if and only if it is of the form $x = (\alpha_{k \mod m})_{k \in \mathbb{Z}}$ with $\alpha_0,\dots,\alpha_{m-1} \in I_{\mathbf{k}}$ that satisfy $A_{\mathbf{k}}(\alpha_0,\alpha_1) = \dots = A_{\mathbf{k}}(\alpha_{m-2}, \alpha_{m-1}) = A_{\mathbf{k}}(\alpha_{m-1},\alpha_0) = 1$.

Moreover, for such an $x$, if $m^\#$ denote the minimal period of $x$, then there are $m^\#$ points in the orbit of $x$. Hence, if $\gamma$ is the orbit in $\mathfrak{P}_{\mathbf{k},m}$ that corresponds to the orbit of $x$ (as explained in Remark \ref{remark:suspension_flow}), there are $m^\#$ points in $\Sigma_{\mathbf{k}}$ that give rise to this orbit. Notice also that the primitive length and length of $\gamma$ are respectively
\begin{equation*}
\sum_{k = 0}^{m^\# - 1} t_{\mathbf{k}}(\sigma^k(x)) \textup{ and }\sum_{k = 0}^{m-1} t_{\mathbf{k}}(\sigma^k(x)) = \frac{m}{m^\#} \sum_{k = 0}^{m^\# - 1} t_{\mathbf{k}}(\sigma^k(x)).
\end{equation*}
It follows that the number of points in $\Sigma_{\mathbf{k}}$ that give rise to the orbit $\gamma$ is $m \frac{T_{\rho_{\mathbf{k}}(\gamma)}^\#}{T_{\rho_{\mathbf{k}}(\gamma)}}$.
\end{remark}

\begin{proof}[Proof of Lemma \ref{lemma:value_trace}]
Let us fix $z \in \mathbb{C}$ and $m \geq 1$. We may write $\mathcal{L}_z^m$ as a matrix of operator $((\mathcal{L}_z^m)_{\alpha,\beta})_{\alpha,\beta \in I_{\mathbf{k}}}$ where 
\begin{equation*}
(\mathcal{L}_z^m)_{\alpha,\beta} = \sum_{\substack{\alpha_0, \dots , \alpha_m \in I_{\mathbf{k}} \\ \alpha_0 = \alpha,\alpha_m = \beta \\A_{\mathbf{k}}(\alpha_0,\alpha_1) = \dots = A_{\mathbf{k}}(\alpha_{m-1}, \alpha_m) = 1}} \mathcal{L}_{z,\alpha_0,\alpha_1} \dots \mathcal{L}_{z,\alpha_{m-1},\alpha_m}.
\end{equation*}
Only the diagonal terms contribute to the trace of $\mathcal{L}_z^m$ and thus we have
\begin{equation*}
\begin{split}
\tr(\mathcal{L}_z^m) & = \sum_{\alpha \in I_{\mathbf{k}}} \tr((\mathcal{L}_z^m)_{\alpha,\alpha}) \\
                     & = \sum_{\substack{\alpha_0, \dots , \alpha_m \in I_{\mathbf{k}} \\ \alpha_0 = \alpha_m \\A_{\mathbf{k}}(\alpha_0,\alpha_1) = \dots = A_{\mathbf{k}}(\alpha_{m-1}, \alpha_m) = 1}} \tr(\mathcal{L}_{z,\alpha_0,\alpha_1} \dots \mathcal{L}_{z,\alpha_{m-1},\alpha_m}).
\end{split}
\end{equation*}
Let us fix $\alpha_0,\dots,\alpha_m \in I_{\mathbf{k}}$ such that $\alpha_0 = \alpha_m$ and $A_{\mathbf{k}}(\alpha_0,\alpha_1) = \dots = A_{\mathbf{k}}(\alpha_{m-1}, \alpha_m) = 1$ and compute the contribution of the corresponding term. Recalling the definition \eqref{eq:definition_entry}, we find that $\mathcal{L}_{z,\alpha_0,\alpha_1} \dots \mathcal{L}_{z,\alpha_{m-1},\alpha_m}$ maps $f$ to 
\begin{equation*}
\begin{split}
& \sum_{a_1 \in P_{\alpha_1},\dots, a_m \in P_{\alpha_m}} \sum_{\ell_1,\dots, \ell_m \in \mathbb{Z}^{n-1}} \sum_{1 \leq p_1,\dots,p_m \leq d}  \\ & \qquad \qquad \langle f, \mathbf{e}_{a_m, \ell_m,p_m} \rangle \prod_{j = 1}^{m-1} \langle L_{z,\alpha_{j},\alpha_{j+1}} \tilde{\mathbf{e}}_{a_{j+1},\ell_{j+1}, p_{j+1}}, \mathbf{e}_{a_{j},\ell_j,p_j} \rangle L_{z,\alpha_0,\alpha_1} \tilde{\mathbf{e}}_{a_1,\ell_1,p_1}.
\end{split}
\end{equation*}
The trace of a rank $1$ operator $e \otimes l$ is $l(e)$ and consequently
\begin{equation*}
\begin{split}
& \tr(\mathcal{L}_{z,\alpha_0,\alpha_1} \dots \mathcal{L}_{z,\alpha_{m-1},\alpha_m}) \\ &  \quad = \sum_{a_1 \in P_{\alpha_1},\dots, a_m \in P_{\alpha_m}} \sum_{\ell_1,\dots, \ell_m \in \mathbb{Z}^{n-1}} \sum_{1 \leq p_1,\dots,p_m \leq d}  \\ & \qquad \qquad  \langle L_{z,\alpha_0,\alpha_1} \tilde{\mathbf{e}}_{a_1,\ell_1,p_1}, \mathbf{e}_{a_m, \ell_m,p_m} \rangle \prod_{j = 1}^{m-1} \langle L_{z,\alpha_{j},\alpha_{j+1}} \tilde{\mathbf{e}}_{a_{j+1},\ell_{j+1}, p_{j+1}}, \mathbf{e}_{a_{j},\ell_j,p_j} \rangle .
\end{split}
\end{equation*}
Using repetitively \eqref{eq:inverse_Fourier}, we find that
\begin{equation*}
\begin{split}
& \tr(\mathcal{L}_{z,\alpha_0,\alpha_1} \dots \mathcal{L}_{z,\alpha_{m-1},\alpha_m}) \\ &  \qquad = \sum_{a \in P_{\alpha_0}} \sum_{\ell \in \mathbb{Z}^{n-1}} \sum_{1 \leq p \leq d} \langle L_{z,\alpha_0,\alpha_1} \dots L_{z,\alpha_{m-1},\alpha_m} \tilde{\mathbf{e}}_{a,\ell,p},\mathbf{e}_{a,\ell,p} \rangle.
\end{split}
\end{equation*}
For $a \in P_{\alpha_0}, \ell \in \mathbb{Z}^{n-1}$ and $1 \leq p \leq d$, we have explicitly
\begin{equation*}
\langle L_{z,\alpha_0,\alpha_1} \dots L_{z,\alpha_{m-1},\alpha_m} \tilde{\mathbf{e}}_{a,\ell,p},\mathbf{e}_{a,\ell,p} \rangle = \int_{\mathbb{R}^{n-1}} h_{a,p}(x) e^{2 i \pi (F_{\alpha_0,\dots,\alpha_m}(x) - x)\cdot \ell} \mathrm{d}x,
\end{equation*}
where 
\begin{equation*}
\begin{split}
h_{a,p}(x) = & \tilde{\theta}_a(F_{\alpha_0,\dots,\alpha_m}(x)) \theta_a(x) \left(\prod_{k = 0}^{m-1} \chi_{\alpha_k,\alpha_{k+1}}(F_{\alpha_0,\dots,\alpha_{k}}(x)) \right)\\ & \qquad \qquad \times \exp( - z \sum_{k= 0}^{m-1} \tau_{\alpha_k,\alpha_{k+1}}(F_{\alpha_0,\dots,\alpha_{k}}(x))) \langle N_{\alpha_0,\dots,\alpha_m}(x) e_p, e_p \rangle.
\end{split}
\end{equation*}
We also used the notations
\begin{equation*}
F_{\alpha_0,\dots,\alpha_{k}} = F_{\alpha_{k-1},\alpha_{k}} \circ \dots \circ F_{\alpha_1,\alpha_2} \circ F_{\alpha_0,\alpha_1}
\end{equation*}
for $k = 0,\dots,m-1$ (with the convention that $F_{\alpha_0,\alpha_0}$ is the identity) and
\begin{equation*}
N_{\alpha_0,\dots, \alpha_{m}}(x) = N_{\alpha_0,\alpha_1}(x) N_{\alpha_1,\alpha_2}( F_{\alpha_0,\alpha_1}(x)) \dots N_{\alpha_{m-1},\alpha_m} (F_{\alpha_0,\dots,\alpha_{m-1}}(x))
\end{equation*}
for $x$ in the support of $\prod_{k = 0}^{m-1} \chi_{\alpha_k,\alpha_{k+1}} \circ F_{\alpha_0,\dots,\alpha_{k}}$.

Notice that if $x_0$ belongs to the support of $h_{a,p}$ then Lemma \ref{lemma:useful_later} implies that $D F_{\alpha_0,\dots,\alpha_{m}} (x_0) - I$ is invertible. Hence, $x \mapsto F_{\alpha_0,\dots,\alpha_m}(x)-x$ induces a diffeomorphism from a neighbourhood of $x_0$ to its image. Consequently, we may cover the support of $h_{a,p}$ by open sets $V_1,\dots, V_q$ such that for $j = 1,\dots,q$ the map $x \mapsto F_{\alpha_0,\dots,\alpha_m}(x) - x$ induces a diffeomorphism from $V_j$ to an open subset\footnote{If $x$ is in the support of $h_{a,p}$ then $x$ and $F_{\alpha_0,\dots,\alpha_{m}}(x)$ belongs respectively to the support of $\theta_a$ and $\tilde{\theta}_a$, and thus to $\overline{B}(a,2\delta/3)$. Hence, $x - F_{\alpha_0,\dots,\alpha_m}(x) \in \overline{B}(0,\delta) \subseteq ]-1,1[^{n-1}$.} $W_j$ of $]-1,1[^{n-1}$. We denote by $g_j : W_j \to V_j$ the inverse of this map. Let then $(w_j)_{1 \leq j \leq q}$ be a family of $C^\infty$ functions on $\mathbb{R}^{n-1}$ such that $\sum_{j = 1}^q w_j \equiv 1$ on a neighbourhood of the support of $h_{a,p}$ and, for $j = 1,\dots,q$, the function $w_j$ is supported in $V_j$.

Fix $a \in P_{\alpha_0}$ and $p \in \set{1,\dots,d}$. For $j = 1,\dots,q$, the change of variable formula implies that for $\ell \in \mathbb{Z}^{n-1}$ we have
\begin{equation*}
\begin{split}
\int_{\mathbb{R}^{n-1}} w_j(x) h_{a,p}(x) & e^{2 i \pi (F_{\alpha_0,\dots,\alpha_m}(x) - x)\cdot \ell} \mathrm{d}x \\ & \qquad = \int_{W_j} w_j \circ g_j(x) h_{a,p} \circ g_j(x) e^{2 i \pi x \cdot \ell} |\det D g_j(x)| \mathrm{d}x.
\end{split}
\end{equation*}
Summing over $\ell \in \mathbb{Z}^{n-1}$, the Poisson summation formula implies that
\begin{equation*}
\begin{split}
\sum_{\ell \in \mathbb{Z}^{n-1}}\int_{\mathbb{R}^{n-1}} w_j(x) h_{a,p}(x) & e^{2 i \pi (F_{\alpha_0,\dots,\alpha_m}(x) - x)\cdot \ell} \mathrm{d}x \\ & \qquad \qquad \qquad \quad = \frac{w_j(g_j(0)) h_{a,p}(g_j(0))}{|\det( I - D F_{\alpha_0,\dots,\alpha_m}(g_j(0))|},
\end{split}
\end{equation*}
where we recall that $w_j$ is supported in $]-1,1[^{n-1}$. If $0$ does not belong to $W_j$, then the right hand side is zero. Notice that there is a fixed point for $F_{\alpha_0,\dots,\alpha_m}$ in $V_j$ if and only if $0$ belongs to $W_j$, and when it is the case $g_j(0)$ is the only fixed point of $F_{\alpha_0,\dots,\alpha_m}$ in $V_j$. Hence, we have
\begin{equation*}
\begin{split}
\sum_{\ell \in \mathbb{Z}^{n-1}}\int_{\mathbb{R}^{n-1}} w_j(x) h_{a,p}(x) & e^{2 i \pi (F_{\alpha_0,\dots,\alpha_m}(x) - x)\cdot \ell} \mathrm{d}x \\ & \qquad \quad = \sum_{x_* \in \Fix(F_{\alpha_{0},\dots,\alpha_m})} \frac{w_j(x_*) h_{a,p}(x_*)}{|\det( I - D F_{\alpha_0,\dots,\alpha_m}(x_*)|}.
\end{split}
\end{equation*}
Summing over $j$, we get
\begin{equation}\label{eq:poisson_summation}
\begin{split}
\sum_{\ell \in \mathbb{Z}^{n-1}} \langle L_{z,\alpha_0,\alpha_1} \dots & L_{z,\alpha_{m-1},\alpha_m} \tilde{\mathbf{e}}_{a,\ell,p}, \mathbf{e}_{a,\ell,p} \rangle \\ & = \sum_{x_* \in \Fix(F_{\alpha_{0},\dots,\alpha_m})} \frac{h_{a,p}(x_*)}{|\det(I - D F_{\alpha_{0}, \dots,\alpha_m}(x_*))|}.
\end{split}
\end{equation}
Let $y_* = \pi_A(p_{\mathbf{k}}((\alpha_{j \mod m})_{j \in \mathbb{Z}}))$. We will show that $\Fix(F_{\alpha_0,\dots,\alpha_m}) = \set{y_*}$. Recalling the definition of $\tau_{\mathbf{k}}$ in \S \ref{subsection:exact_counting} and the definition of the $\tau_{\alpha,\beta}$'s, we find that $y_* \in \Fix(F_{\alpha_0,\dots,\alpha_m})$. Reciprocally, if $x_* \in \Fix(F_{\alpha_0,\dots,\alpha_m})$, then we find that for every $j \in \mathbb{Z}$ we have $\phi_{t_j}(x_*) \in D_{i(\alpha_{j \mod m})}$ and $\phi_{u_j}(y_*) \in D_{i(\alpha_{j \mod m})}$ where (with periodic indexing of the $\alpha_k$'s)
\begin{equation*}
t_j = \begin{cases} \sum_{k= 0}^{j-1} \tau_{\alpha_{k},\alpha_{k+1}}(F_{\alpha_0,\dots,\alpha_k}(x_*)) & \textup{ if } j \geq 0 \\
                    - \sum_{k= 0}^{-j-1} \tau_{\alpha_{-k-1},\alpha_{-k}}((F_{\alpha_{-k-1},\alpha_{-k}} \circ \dots \circ F_{\alpha_{-1},\alpha_0})^{-1}(x_*))   & \textup{ if }  j < 0. \end{cases}
\end{equation*}
The $u_j$'s are defined similary with $x_*$ replaced by $y_*$. Let $\nu$ be given by Proposition \ref{proposition:expansivity} with $\epsilon = T_0/2$ and recall that we use a Markov partition of size less than $\nu$ and that the $\tau_{\alpha,\beta}$'s take value in $(0,2\kappa) \subseteq (0,\nu)$. For every $j \in \mathbb{Z}$, we have $0 < t_{j+1} - t_j < \nu, 0 < u_{j+1} - u_j < \nu$ and $d(\phi_{t_j}(x_*),\phi_{u_j}(x_*)) < \nu$. Hence, there is $t \in [-T_0/2,T_0/2]$ such that $y_* = \phi_t(x_*)$, but then Lemma \ref{lemma:better_partition} (see also Remark \ref{remark:really_small}) implies that $t =0$ and thus $x_* = y_*$.

Let then $\gamma_{\alpha_0,\dots,\alpha_m}$ denote the periodic orbit of the suspension flow $\tilde{\psi}_{\mathbf{k}}$ of length $\sum_{k= 0}^{m-1} \tau_{\alpha_k,\alpha_{k+1}}\circ F_{\alpha_0,\dots,\alpha_k}(y_*)$ passing through the projection of the point $((\alpha_{k \mod m})_{k \in \mathbb{Z}},0)$. Since the points $y_*, F_{\alpha_0,\alpha_1}(y_*), \dots, F_{\alpha_0,\dots,\alpha_{m-1}}(y_*)$ belong respectively to $R_{\alpha_0,\alpha_1},R_{\alpha_1,\alpha_2},\dots, R_{\alpha_{m-1},\alpha_m}$, we have
\begin{equation*}
\left(\prod_{k = 0}^{m-1} \chi_{\alpha_k,\alpha_{k+1}}(F_{\alpha_0,\dots,\alpha_{k}}(y_*)) \right) = 1.
\end{equation*}
The lengths of $\gamma_{\alpha_0,\dots,\alpha_m}$ and its image by $\rho_{\mathbf{k}}$ are the same and thus:
\begin{equation*}
\sum_{k= 0}^{m-1} \tau_{\alpha_k,\alpha_{k+1}}(F_{\alpha_0,\dots,\alpha_{k}}(y_*)) = T_{\rho_{\mathbf{k}}(\gamma_{\alpha_0,\dots,\alpha_m})}.
\end{equation*}
Since $F_{\alpha_0,\dots,\alpha_m}$ is induced by the flow on a disk transverse to $\rho_{\mathbf{k}}(\gamma_{\alpha_0,\dots,\alpha_m})$, we have
\begin{equation*}
\det(I - D(F_{\alpha_0,\dots,\alpha_m})(y_*)) = \det(I - \mathcal{P}_{\rho_{\mathbf{k}}(\gamma_{\alpha_0,\dots, \alpha_{m}})}).
\end{equation*}
Recalling the definition \eqref{eq:action_fiber}, we find
\begin{equation*}
N_{\alpha_0,\dots,\alpha_m}(y_*) = \left(H_{i(\alpha_0),y_*}^{-1} \Phi_{\rho_{\mathbf{k}}(\gamma_{\alpha_0,\dots,\alpha_m})} H_{i(\alpha_0),y_*} \right)^\top.
\end{equation*}
Hence, summing \eqref{eq:poisson_summation} over $a$ and $p$, we get
\begin{equation*}
\tr(\mathcal{L}_{z,\alpha_0,\alpha_1} \dots \mathcal{L}_{z,\alpha_{m-1},\alpha_m}) = \frac{e^{-zT_{\rho_{\mathbf{k}}(\gamma_{\alpha_0,\dots,\alpha_m})}} \tr(\Phi_{\rho_{\mathbf{k}}(\gamma_{\alpha_0,\dots,\alpha_m})})}{|\det(I - \mathcal{P}_{\rho_{\mathbf{k}}(\gamma_{\alpha_0,\dots, \alpha_{m}})})|.}
\end{equation*}
Here, we used that $\sum_{a \in P_{\alpha_0}} \tilde{\theta}_a(y_*) \theta_a(y_*) = \sum_{a \in P_{\alpha_0}} \theta_a(y_*) = 1$ since $y_*$ belongs to $R_{\alpha_0}$. Summing over all cycles of length $m$ in the graph defining the shift $\Sigma_{\mathbf{k}}$ the result follows. Indeed, the number of different $(\alpha_0,\dots,\alpha_m)$'s that give rise to a given $\gamma \in \mathfrak{P}_{\mathbf{k},m}$ is the minimal period of $(\alpha_{j \mod m})_{j \in \mathbb{Z}}$, which is also $m \frac{T_{\rho_{\mathbf{k}}(\gamma)}}{T^\#_{\rho_{\mathbf{k}}(\gamma)}}$, see Remark \ref{remark:number_antecedents}.
\end{proof}

\begin{lemma}\label{lemma:identification_determinant}
For every $z \in \mathbb{C}$ with $\re z \gg 1$, we have $d_{\mathbf{k}}(z) = \det(I - \mathcal{L}_z)$.
\end{lemma}

\begin{proof}
For every $z \in \mathbb{C}$, if $w \in \mathbb{C}$ is sufficiently small, we have
\begin{equation*}
\det( I - w \mathcal{L}_z) = \exp( - \sum_{m \geq 1} \frac{w^m}{m} \tr(\mathcal{L}_z^m)).
\end{equation*}
It follows from Lemma \ref{lemma:value_trace} that if $\re z \gg 1$ then the series in the right hand side converges whenever $|w| \leq 2$ and thus the equality still holds in this range of $w$ by the analytic continuation principle. Hence, if $\re z \gg 1$, we can take $w = 1$ in the equality above and get
\begin{equation*}
\begin{split}
\det(I - \mathcal{L}_z) & = \exp( - \sum_{m \geq 1} \frac{1}{m} \sum_{\gamma \in \mathfrak{P}_{\mathbf{k},m}} m \frac{\tr(\Phi_{\rho_{\mathbf{k}}(\gamma)})}{|\det(I - \mathcal{P}_{\rho_{\mathbf{k}}(\gamma)})|} e^{- z T_{\rho_{\mathbf{k}}(\gamma)}})\\
       & = \exp( - \sum_{\gamma \in \mathfrak{P}_{\mathbf{k}}} \frac{T_{\rho_{\mathbf{k}}(\gamma)}^{\#}}{T_{\rho_{\mathbf{k}}(\gamma)}} \frac{\tr(\Phi_{\rho_{\mathbf{k}}(\gamma)})}{|\det(I - \mathcal{P}_{\rho_{\mathbf{k}}(\gamma)})|} e^{- z T_{\rho_{\mathbf{k}}(\gamma)}}) = d_ {\mathbf{k}}(z).
\end{split}
\end{equation*}
\end{proof}

We can now end the proof of Lemma \ref{lemma:individual_estimates}, and thus of Theorems \ref{theorem:bound_resonances} and \ref{theorem:bound_determinant}.

\begin{proof}[Proof of Lemma \ref{lemma:individual_estimates}]
This is a consequence of Lemmas \ref{lemma:abstract_determinant} and \ref{lemma:identification_determinant}.
\end{proof}

\appendix

\section{Bound on entire functions}\label{appendix:bound_entire_functions}

We gathered in this appendix some merely technical results on entire functions.

\begin{lemma}\label{lemma:simple_bound}
Let $(\lambda_j)_{j \geq 0}$ be a sequence of complex numbers. Let $\alpha > 0$. Assume that
\begin{equation*}
\# \set{ j \in \mathbb{N} : |\lambda_j| \geq r } \underset{r \to 0}{=} \mathcal{O}(|\log r|^\alpha).
\end{equation*}
Then the infinite product
\begin{equation}\label{eq:little_infinite_product}
\prod_{j \geq 0} (1 - z \lambda_j)
\end{equation}
converges for $z \in \mathbb{C}$ to an entire function $f(z)$ that satisfies
\begin{equation*}
|f(z)| \leq C \exp( C\log(1+|z|)^{1 + \alpha})
\end{equation*}
for some $C > 0$ and every $z \in \mathbb{C}$.
\end{lemma}

\begin{proof}
Let $C$ be a constant such that 
\begin{equation*}
\# \set{ j \in \mathbb{N} : |\lambda_j| \geq r } \leq C(1 + |\log r|^{\alpha})
\end{equation*}
for every $r > 0$. Now, for $r \in (0,1)$, we have
\begin{equation*}
\begin{split}
\sum_{\substack{j \geq 0 \\ |\lambda_j| \leq r}} |\lambda_j| & = \sum_{k \geq 0} \sum_{\substack{j \geq 0 \\  \frac{r}{2^{k+1}} < |\lambda_j| \leq \frac{r}{2^k}}} |\lambda_j| \leq \sum_{k \geq 0} \frac{C r (1 + |\log(2^{-k-1}r)|^\alpha)}{2^k} \\
       & \leq C r (1 + |\log r|^\alpha),
\end{split}
\end{equation*}
where the constant $C$ may be larger on the second line. In particular, we find that the series $\sum_{j \geq 0} |\lambda_j|$ converges, which proves that the infinite product \eqref{eq:little_infinite_product} converges to an entire function $f(z)$.

Let $A = \sup\limits_{j \geq 0} |\lambda_j|$. For $z \in \mathbb{C}$ with $|z| \geq 1$, we have
\begin{equation*}
\begin{split}
|f(z)| & \leq \prod_{\substack{j \geq 0 \\ |\lambda_j| \leq |z|^{-1}}} (1 + |z||\lambda_j|) \prod_{\substack{j \geq 0 \\ |\lambda_j| > |z|^{-1}}} (1 + |z||\lambda_j|) \\
    & \leq \exp(|z| \sum_{\substack{j \geq 0 \\ |\lambda_j| \leq |z|^{-1}}} |\lambda_j|) (2 A |z|)^{\# \set{j \geq 0 : |\lambda_j| > |z|^{-1}}}\\
    & \leq \exp(C(1 + |\log z |^\alpha + C (1 + |\log z|^\alpha)|\log(2A |z|)|)).
\end{split}
\end{equation*}
The result follows.
\end{proof}

\begin{lemma}\label{lemma:bound_entire_function}
Let $f$ and $g$ be two entire functions. Let $\alpha > 0$. Assume that there is a constant $C > 0$ such that for every $z \in \mathbb{C}$ we have
\begin{equation}\label{eq:almost_finite_order}
|f(z)| \leq C \exp(C |z|^\alpha) \textup{ and } |g(z)| \leq C \exp(C |z|^\alpha).
\end{equation}
Assume that $g$ is non-identically equal to $0$ and that the function $f/g$ is entire (i.e. that the zeros of $g$ are contained in the zeros of $f$). Then 
\begin{equation}\label{eq:general_bound_zeros}
\#\set{ z \in \mathbb{C}: z \textup{ zero of } f/g \textup{ and } |z| \leq r} \underset{r \to + \infty}{=} \mathcal{O}(r^\alpha),
\end{equation}
where the zeros are counted with multiplicities. Moreover, there is a constant $C_0 > 0$ such that for every $z \in \mathbb{C}$, we have
\begin{equation*}
\left| \frac{f(z)}{g(z)}\right| \leq \begin{cases} C_0 \exp(C_0 |z|^\alpha) & \textup{ if } \alpha \textup{ is not an integer },\\
C_0 \exp(C_0 |z|^\alpha\log(1+|z|)) & \textup{ if } \alpha \textup{ is an integer.} \end{cases}
\end{equation*}
\end{lemma}

\begin{proof}
Let $p$ be the integer part of $\alpha$. The orders of $f$ and $g$ are less than $\alpha$ and thus we may apply the Weierstrass Factorization Theorem to $f$ and $g$, which gives:
\begin{equation*}
f(z) = e^{P(z)} z^m \prod_{j \geq 1} W_p(\frac{z}{\lambda_j}) \textup{ and } g(z) = e^{Q(z)} z^{m'} \prod_{j \geq 1} W_p(\frac{z}{\mu_j})
\end{equation*}
where $P$ and $Q$ are polynomials of degree at most $p$, the numbers $m$ and $m'$ are the multiplicity of $0$ as a zero of $f$ and $g$ respectively, $(\lambda_j)_{j \geq 1}$ and $(\mu_{j})_{j \geq 1}$ are the non-zero zeroes of $f$ and $g$ respectively, and the factor $W_p$ is defined by
\begin{equation*}
W_p(z) = (1 - z) \exp\left(\sum_{k = 1}^p \frac{z^k}{k}\right) = \exp\left(- \sum_{k = p +1}^{+ \infty} \frac{z^k}{k}\right).
\end{equation*}
The last expression is only valid for $|z| < 1$. Let $(\tilde{\lambda}_j)_{j \geq 1}$ be the sequence $(\lambda_j)_{j \geq 1}$ from which we removed the $\mu_j$'s. We have then that
\begin{equation}\label{eq:f_sur_g}
\frac{f(z)}{g(z)} = e^{P(z) - Q(z)} z^{m - m'} \prod_{j \geq 1} W_p(\frac{z}{\tilde{\lambda}_j}).
\end{equation}
To estimate this product, let us introduce:
\begin{equation*}
n(r) = \# \set{ j \geq 1 : |\lambda_j| \leq r} \textup{ and } \widetilde{n}(r) =\set{ j \geq 1 : |\tilde{\lambda}_j| \leq r}.
\end{equation*}
Introduce also $\tilde{f}(z) = z^{-m} f(z)$. Since $\tilde{f}(0) \neq 0$, Jensen's formula gives
\begin{equation*}
\log |\tilde{f}(0)| + \int_{0}^{r} \frac{n(t)}{t}\mathrm{d}t = \frac{1}{2 \pi} \int_0^{2 \pi} \log |\tilde{f}(r e^{i \theta})| \mathrm{d}\theta.
\end{equation*}
Thus, we have
\begin{equation*}
n(r) \leq 2 \int_{r}^{2r} \frac{n(t)}{t} \mathrm{d}t \leq - 2 \log |\tilde{f}(0)| + \frac{1}{\pi} \int_0^{2 \pi} \log |\tilde{f}(r e^{i \theta})| \mathrm{d}\theta.
\end{equation*}
We deduce then from \eqref{eq:almost_finite_order} that $n(r) \underset{r \to + \infty}{=} \mathcal{O}(r^{\alpha})$ and thus that $\widetilde{n}(r) \underset{r \to + \infty}{=} \mathcal{O}(r^{\alpha})$. We just proved \eqref{eq:general_bound_zeros}.

Notice that, in order to obtain the bound that we announced on $f/g$, we only need to estimate the product $\prod_{j \geq 1} W_p(\tilde{\lambda}_j^{-1} z)$ (the other factors clearly satisfy this bound). Assume first that $\alpha$ is not an integer. Let us notice that there is a constant $C$ such that
\begin{equation*}
|W_p(z)| \leq \exp(C |z|^p) \textup{ for every } z \in  \mathbb{C} \setminus \mathbb{D}(0,\frac{1}{2})
\end{equation*}
and
\begin{equation*}
|W_p(z)| \leq \exp(C |z|^{p+1}) \textup{ for every } z \in  \mathbb{D}(0,\frac{1}{2}).
\end{equation*}
We can split the product accordingly. We estimate first
\begin{equation*}
\left| \prod_{\substack{j \geq 1 \\ |\tilde{\lambda}_j| \leq 2|z|}} W_p(\tilde{\lambda}_j^{-1} z) \right| \leq \exp \left(C |z|^p\sum_{\substack{j \geq 1 \\ |\tilde{\lambda}_j| \leq 2|z|}} |\tilde{\lambda}_j|^{-p}\right).
\end{equation*}
To bound the sum in the exponential, let us consider the smallest integer $T$ such that $2^{T+1} \geq 2 |z|$. Then, we have
\begin{equation}\label{eq:packets}
\begin{split}
\sum_{\substack{j \geq 1 \\ |\tilde{\lambda}_j| \leq 2|z|}} |\tilde{\lambda}_j|^{-p} & \leq \sum_{\substack{j \geq 1 \\ 0 \leq |\tilde{\lambda}_j| \leq 1}} |\tilde{\lambda}_j|^{-p} + \sum_{t = 0}^T \sum_{\substack{j \geq 1 \\ 2^t \leq |\tilde{\lambda}_j| \leq 2^{t+1}}} |\tilde{\lambda}_j|^{-p} \\
& \leq C + \sum_{t = 0}^T 2^{-tp} \widetilde{n}(2^{t+1}) \leq C(1 + \sum_{t =0}^T 2^{t(\alpha - p)}) \leq C 2^{T(\alpha - p)} \\
& \leq C |z|^{\alpha -p},
\end{split}
\end{equation}
where the constant $C$ may change from one line to another and within the second line, but never depends on $|z|$. Here, we used that $\alpha > p$ since $\alpha$ is not an integer. We also assumed in the last line that $|z| \geq 1$. Hence, we find, with a new constant $C$, that
\begin{equation}\label{eq:small_lambda}
\left| \prod_{\substack{j \geq 1 \\ |\tilde{\lambda}_j| \leq 2|z|}} W_p(\tilde{\lambda}_j^{-1} z) \right| \leq \exp(C |z|^{\alpha})
\end{equation}
for $|z|$ large. We deal now with the other factors:
\begin{equation*}
\left| \prod_{\substack{j \geq 1 \\ |\tilde{\lambda}_j| > 2|z|}} W_p(\tilde{\lambda}_j^{-1} z) \right| \leq \exp(C |z|^{p+1} \sum_{\substack{j \geq 1 \\ |\tilde{\lambda}_j| > 2|z|}} |\tilde{\lambda}_j|^{-p-1}).
\end{equation*}
Summing by packets as above, we find that 
\begin{equation*}
\sum_{\substack{j \geq 1 \\ |\tilde{\lambda}_j| > 2|z|}} |\tilde{\lambda}_j|^{-p-1} \underset{|z| \to + \infty}{=} \mathcal{O}(|z|^{\alpha - p - 1})
\end{equation*}
and it follows, with a new constant $C$ and for $|z|$ large that
\begin{equation}\label{eq:large_lambda}
\left| \prod_{\substack{j \geq 1 \\ |\tilde{\lambda}_j| > 2|z|}} W_p(\tilde{\lambda}_j^{-1} z) \right| \leq \exp(C |z|^{\alpha}).
\end{equation}

Putting our two estimates \eqref{eq:small_lambda} and \eqref{eq:large_lambda} together and recalling \eqref{eq:f_sur_g}, the result is proven when $\alpha$ is not an integer. When it is an integer, we do the same computation, but we see in the estimate for the first factor (corresponding to small $\tilde{\lambda}_j$) that a factor $\log |z|$ appears. Indeed, we have $\alpha = p$ and thus in \eqref{eq:packets} we get $\sum_{t = 0}^T 2^t(\alpha-p)= T+1$ and the order of magnitude of $T$ is $\log(1 + |z|)$. Hence, in \eqref{eq:small_lambda} the factor $|z|^\alpha$ in the exponential is replaced by $|z|^\alpha\log(1+|z|)$.
\end{proof}

\bibliographystyle{alpha}
\bibliography{zeta_gevrey.bib}

\end{document}